%% file: main.tex
\documentclass[table,xcdraw]{article}
\usepackage[utf8]{inputenc}
\usepackage[margin=1in]{geometry}
\usepackage{amsfonts}
\usepackage{amsmath}
\usepackage[shortlabels]{enumitem}
\usepackage{amsthm}
\usepackage{amssymb}
\usepackage[style=alphabetic]{biblatex}
\renewbibmacro{in:}{}
\usepackage{xr}
\usepackage{fancyhdr}
\usepackage{subcaption}
\usepackage{amsmath}
\usepackage{blkarray}
\usepackage{bm}
\usepackage{hyperref}
\hypersetup{
    colorlinks=true,
    linkcolor=blue,
    filecolor=magenta,      
    urlcolor=cyan,
    citecolor=blue
}
\bibliography{sources.bib}

\usepackage{url}
\usepackage{float}
\usepackage{graphicx}
\usepackage{circuitikz}
\usepackage{quiver}
\usepackage{mathrsfs}
\usepackage{abstract}
\usepackage{xcolor}

\title{The Green correspondence for ${SL_2(\mathbb{F}_p)}$}
\author{Denver-James Marchment}
\date{}

\newtheorem{thm}{Theorem}[section]

\newtheorem*{thm*}{Theorem}
\newtheorem{cor}[thm]{Corollary}

\newtheorem{lem}[thm]{Lemma}
\newtheorem{prop}[thm]{Proposition}

\theoremstyle{definition}
\newtheorem{defn}[thm]{Definition}

\newtheorem{rem}[thm]{Remark}

\newcommand{\angel}[1]{\langle #1 \rangle}

\newtheorem*{rem*}{Remark}

\newcommand{\thh}{\text{th}}

\newcommand{\Mod}[1]{\ \mathrm{mod}\ (#1)}
\newcommand{\Modwb}[1]{\ \mathrm{mod}\ #1}

\newcommand{\Res}{\textrm{\normalfont Res}}
\newcommand{\Ind}{\textrm{\normalfont Ind}}

\newcommand{\soc}{\textrm{{\normalfont Soc}}}
\newcommand{\rad}{\textrm{{\normalfont Rad}}}
\newcommand{\Top}{\textrm{{\normalfont Top}}}

\DeclareMathOperator{\Span}{\normalfont Span}
\newcommand{\littleb}{\textrm{{\ss}}}
\usepackage{mathtools}
\DeclarePairedDelimiter\ceil{\lceil}{\rceil}
\DeclarePairedDelimiter\floor{\lfloor}{\rfloor}
\usepackage{circuitikz}
\usepackage{float}
\usepackage{enumitem}
\usepackage{underoverlap}\usepackage[justification=centering]{caption}
\begin{document}

\maketitle

\begin{abstract}
    \input{preliminaries/abstract}
\end{abstract}

\section{Introduction}
\input{preliminaries/introduction}

\section*{Acknowledgements}
\input{acknowledgements}

\section{Notation}
\input{preliminaries/notation}

\section{The indecomposables and block structure}
\input{indecomposables/indecomposables}

\section{Describing the Green correspondence}
\input{greencorrespondence/greencorrespondence}

\section{Induction and restriction}
\input{induction-and-restriction/induction-and-restriction}

\section{Bibliography}
\input{bibliography.tex}
\end{document}

%% file: preliminaries/abstract.tex
Let ${p > 2}$ be an odd prime and ${G = SL_2(\mathbb{F}_p)}$. Denote the subgroup of upper triangular matrices as $B$. Finally, let ${\mathbb{F}}$ be an algebraically closed field of characteristic ${p}$. The Green correspondence gives a bijection between the non-projective indecomposable ${\mathbb{F}[G]}$ modules and non-projective indecomposable ${\mathbb{F}[B]}$ modules, realised by restriction and induction. In this paper, we start by recalling a suitable description of the non-projective indecomposable modules for these group algebras. Next, we explicitly describe the Green correspondence bijection by pinpointing the modules' position on the Stable Auslanden-Reiten quivers. Finally, we obtain two corollaries in terms of these descriptions: formulae for lifting the ${\mathbb{F}[B]}$ module decomposition of an ${\mathbb{F}[G]}$ module, and a complete description of ${\Ind_B^G}$ and ${\Res^G_B}$.

%% file: preliminaries/introduction.tex
The special linear group is a classical group that is used within many areas of mathematics. One such area is modular representation theory, where one considers the representations of a finite group whose order is divisible by the characteristic of the base field. In particular, the representation theory of the group ${G = SL_2(\mathbb{F}_p)}$ ($p$ an odd prime) over a field ${\mathbb{F} = \overline{\mathbb{F}}}$ of characteristic $p$ is commonly studied in the literature (see for example \cite[\S \RN{1}, \RN{3}, \RN{5}]{localrep}, \cite[\S 6, 11]{webb}, \cite[\S \RN{4}, \RN{5}]{schneider}, \cite{bonnafe}). In this example, one nice property of $G$ is that its non-projective indecomposable representations are in one-to-one correspondence with those of the 
subgroup $B$ of upper-triangular matrices of $G$. In particular, let ${M}$ be a non-projective indecomposable ${\mathbb{F}[G]}$ module. Then, ${\Res^G_B(M) \cong N\oplus Q}$, where $N$ is non-projective indecomposable and ${Q}$ is projective. Furthermore, ${\Ind^G_B(N)\cong M\oplus P}$, where $P$ is projective. This is a special case of the \textit{Green correspondence} (see \cite[\S \RN{3}.10, Thm. 1]{localrep}). While the group $G$ is used as a canonical example for the Green correspondence in the sources mentioned, this work provides the explicit computations that describe the bijection, offering a level of detail not yet fully explored in previous studies. \par
We get two main Theorems as a result of our considerations. The first concerns taking an ${\mathbb{F}[G]}$ module for which (1) the decomposition of ${\Res^G_B(M)}$ as a direct sum of indecomposable ${\mathbb{F}[B]}$ modules is known and (2) the composition factors of $M$ are known. We are then able to give the full decomposition of $M$ as a direct sum of indecomposable ${\mathbb{F}[G]}$ modules (Theorem \ref{lifting decomposition}). The second result is the full description of induction and restriction of ${\mathbb{F}[B]}$ and ${\mathbb{F}[G]}$ modules (Corollary \ref{ind^G_B(U_{a,b})}, Lemma \ref{Res^G_B(P_{V_t})}, Theorem \ref{restriction of non-projective indecomposable F[G] module}). The author intends to use the results of this paper to extend those of \cite{lucas}, in which the equivariant decomposition of the canonical representation of the Drinfeld curve is given. Namely, the equivariant decomposition of the space of globally holomorphic polydifferentials of the Drinfeld curve will be computed. \par
Our investigation begins by recalling the indecomposable ${\mathbb{F}[B]}$ modules (Proposition \ref{parameterisation of F[B] modules}), as well as the block decomposition of the group algebras ${\mathbb{F}[B]}$ and ${\mathbb{F}[G]}$ (Propositions \ref{block structure for F[B]}, \ref{block structure for F[G]}). Next, we use results from \cite{janusz} to give a description of the non-projective indecomposable ${\mathbb{F}[G]}$ modules based on Brauer tree walks (Definition \ref{parameterisation of the non-projective indecomposable F[G] modules}). From here, we note that the Green correspondence provides an isomorphism between the \textit{Stable Auslanden-Reiten quivers} of the blocks of ${\mathbb{F}[B]}$ and ${\mathbb{F}[G]}$. It is through an exploration of the structure of the Stable A.R. quivers of these blocks that we explicitly describe the Green correspondence. The two Stable A.R. quivers of the blocks of ${\mathbb{F}[B]}$ have a very simple description (Corollary \ref{stable a.r. quivers for F[B]}); however, this is not the case for the two Stable A.R. quivers of the blocks of ${\mathbb{F}[G]}$. We use results from \cite{bleher-chinburg} to narrow down the position of the non-projective indecomposable ${\mathbb{F}[G]}$ modules in the corresponding quiver (Corollaries \ref{small correspondence}, \ref{boundaries of F[G] modules}, \ref{small correspondence extended}). Finally, we utilize this information together with dimension matching (Proposition \ref{dimension of F[G] module modulo p}) to describe the Green correspondent of a non-projective indecomposable ${\mathbb{F}[B]}$ module as a walk on one of the Brauer trees for ${\mathbb{F}[G]}$ (Theorem \ref{green correspondence G}).

%% file: acknowledgements.tex
I would like to express sincere gratitude to Bernhard K\"{o}ck, for his many hours of unwavering guidance and support. I also want to thank Frauke M. Bleher, who pointed me to her work \cite{bleher-chinburg} with Ted Chinburg. Without the contributions made by either of these people, this work would not have been possible.

%% file: preliminaries/notation.tex
We give a list here of notations used throughout this paper.
\begin{figure}[H]
        \input{preliminaries/Tables/notation}
\end{figure}
In addition, whenever we refer to the Grothendieck group ${K_0(\lambda)}$ of a finite-dimensional $k$-algebra ${\lambda}$, we mean:
$$
{
    K_0(\lambda) = \mathbb{Z}\left\{[X]: X \in \lambda-\text{mod}\right\} / \{[A] - [B] + [C]:\ \exists\text{ s.e.s }0\rightarrow A\rightarrow B\rightarrow C\rightarrow 0\}
}
$$
where ${[X]}$ denotes the isomorphism class of a ${\lambda}$-module $X$.

%% file: preliminaries/Tables/notation.tex
\begin{table}[H]
\centering
\begin{tabular}{|c|c|c|}
\hline
Notation                        & Explanation                                                                                                                                                        & Reference      \\ \hline
$p$                             & An odd prime.                                                                                                                                                      & N/A.           \\ \hline
${\mathbb{F}}$                  & \begin{tabular}[c]{@{}c@{}}An algebraically closed field of characteristic $p$.\end{tabular}                                                                     & N/A.           \\ \hline
${\zeta}$                       & A primitive root modulo $p$.                                                                                                                                       & N/A.           \\ \hline
${G}$                           & The group ${SL_2(\mathbb{F}_p)}$.                                                                                                                                  & N/A.           \\ \hline
${U = \langle g \rangle}$       & \begin{tabular}[c]{@{}c@{}}The subgroup of upper uni-triangular matrices of $G$.\end{tabular}                                                                    & N/A.           \\ \hline
${T = \langle \lambda \rangle}$ & \begin{tabular}[c]{@{}c@{}}The subgroup of diagonal matrices of $G$.\end{tabular}                                                                                & N/A.           \\ \hline
${B = U\rtimes_{\chi} T}$       & \begin{tabular}[c]{@{}c@{}}The subgroup of upper triangular matrices of ${G}$.\end{tabular}                                                                      & N/A.           \\ \hline
${U_{a,b}}$                     & The indecomposable ${\mathbb{F}[B]}$ modules.                                                                                                                      & \ref{parameterisation of F[B] modules}. \\ \hline
${S_a = U_{a,1}}$               & The simple ${\mathbb{F}[B]}$ modules.                                                                                                                              &\ref{parameterisation of F[B] modules}. \\ \hline
${\littleb_0,\littleb_1}$       & The two blocks of ${\mathbb{F}[B]}$.                                                                                                                               &\ref{block structure for F[B]}.   \\ \hline
${V_t}$                         & The simple ${\mathbb{F}[G]}$ modules.                                                                                                                              &\ref{simple F[G] modules}. \\ \hline
${\mathcal{B}_0,\mathcal{B}_1}$ & \begin{tabular}[c]{@{}c@{}}The two non semi-simple blocks of ${\mathbb{F}[G]}$.\end{tabular}                                                                     & \ref{block structure for F[G]}.   \\ \hline
${M(i,l,s,\epsilon)}$           & \begin{tabular}[c]{@{}c@{}}The indecomposable ${\mathbb{F}[G]}$ module associated with a given walk on the Brauer tree.\end{tabular}                                   &\ref{parameterisation of the non-projective indecomposable F[G] modules}.  \\ \hline
${\Gamma^s_{\Lambda}}$          & The Stable A.R. quiver of ${\Lambda}$.                                                                                                                             & N/A.           \\ \hline
${V_{a,b}}$                     & The Green correspondent of ${U_{a,b}}$.                                                                                                                            & N/A.           \\ \hline
${\Omega^2(-)}$                     & The Heller operator applied twice.                                                                                                                            & N/A.           \\ \hline
${\mathbf{1}_A(x)}$                   & \begin{tabular}[c]{@{}c@{}}The indicator function (returns $1$ if ${x \in A}$, and $0$ otherwise).\end{tabular}                                                      &\S 3.   \\ \hline
${c_{a,b,t}}$                   & \begin{tabular}[c]{@{}c@{}}The multiplicity of ${V_t}$ as a composition factor of ${V_{a,b}}$.\end{tabular}                                                      &\ref{c_{a,b}(t) explicit description and c_{a,b,t}}.   \\ \hline
${\Gamma,\ \mathscr{B}}$        & \begin{tabular}[c]{@{}c@{}}The Cartan matrix ${\Gamma}$ of ${\mathcal{B}_i}$ and its inverse ${\mathscr{B} = \Gamma^{-1}}$.\end{tabular}                                                    &\ref{inverse of the cartan matrix}.       \\ \hline
${\alpha_t}$                    & \begin{tabular}[c]{@{}c@{}}The multiplicity of ${V_t}$ as a composition factor of some projective ${\mathbb{F}[G]}$ module.\end{tabular} & \S 4.           \\ \hline
${n_i}$                         & \begin{tabular}[c]{@{}c@{}}The multiplicity of ${P_{V_{t}}}$ (the projective cover of ${V_t}$) as a summand of some \\${\mathbb{F}[G]}$ module.\end{tabular}       & N/A.           \\ \hline
${\ell_t}$                      & \begin{tabular}[c]{@{}c@{}}The multiplicity of ${V_t}$ as a composition factor of some ${\mathbb{F}[G]}$ module.\end{tabular}                                & \S 4.           \\ \hline
${n_{a,b}}$                     & \begin{tabular}[c]{@{}c@{}}The multiplicity of ${U_{a,b}}$ as a summand of some ${\mathbb{F}[B]}$ module.\end{tabular}                                           & N/A.           \\ \hline
${\theta_{a,b,c}}$              & The multiplicity of ${S_c}$ as a composition factor of ${U_{a,b}}$.                                                                                                & \S 4.           \\ \hline
${\ell_{a,b,t}}$                & \begin{tabular}[c]{@{}c@{}}The multiplicity of ${V_t}$ as a composition factor of ${\Ind^G_B(U_{a,b})}$.\end{tabular}                                            & \S 4.           \\ \hline
${\kappa_c}$                    & \begin{tabular}[c]{@{}c@{}}The multiplicity of ${S_c}$ as a composition factor of some projective ${\mathbb{F}[B]}$ module.\end{tabular} & \S 4.           \\ \hline
${\mathbf{\gamma}}$,\ ${\mathbf{\delta}}$            & The Cartan matrix ${\mathbf{\gamma}}$ of ${\littleb_i}$ and its inverse ${\mathbf{\delta} = \mathbf{\gamma}^{-1}}$.                                                                                                                               & \S 4.           \\ \hline
${\Theta_{i,l,s,t}}$            & The multiplicity of ${V_t}$ as a composition factor of ${M(i,l,s,\epsilon)}$.                                                                                      & \S 4.           \\ \hline
\end{tabular}
\end{table}

%% file: indecomposables/indecomposables.tex
This section primarily gathers established results from the literature and applies them to the specific context of our groups. We begin by recalling a parameterisation of the indecomposable ${\mathbb{F}[B]}$ modules.
\begin{prop}
    \label{parameterisation of F[B] modules}
    For ${a \in \mathbb{Z}}$, let ${S_a}$ denote the one dimensional vector space over ${\mathbb{F}}$ on which $g$ acts trivially and ${\lambda}$ acts as multiplication by ${\zeta^a}$. Then ${S_0, S_1, ..., S_{p-2}}$ gives all the simple modules for ${\mathbb{F}[B]}$. Next, let ${1\leq b \leq p}$. Then there exists a uniserial module ${U_{a,b}}$ of dimension $b$ whose composition factors are given by ${S_a, S_{a+2}, S_{a+4}, ...}$ in ascending order. This describes all indecomposable ${\mathbb{F}[B]}$ modules. The projective indecomposables are given by ${U_{a,p}}$.
\end{prop}
\begin{proof}
    The projective indecomposable ${\mathbb{F}[B]}$ modules and their composition factors follow from \cite[\S \RN{2}.5, pp. 35-37]{localrep}, \cite[\S \RN{2}.5, pp. 37, ex. 3]{localrep}. It follows from \cite[\S \RN{2}.6, pp. 42-43]{localrep} that all indecomposable ${\mathbb{F}[B]}$ modules are uniserial and obtained as homomorphic images of projective indecomposables, giving the full description.
\end{proof}
\begin{prop}
    \label{block structure for F[B]}
    The group algebra ${\mathbb{F}[B]}$ consists of two blocks ${\mathbb{F}[B] \cong \littleb_0 \times \littleb_1}$. They both have cyclic defect, and have the following Brauer trees:
    \begin{figure}[H] 
        \centering
        \input{indecomposables/tikz/brauer-tree-B-b0}
        \input{indecomposables/tikz/brauer-tree-B-b1}
        \caption{{The Brauer trees ${\littleb_0}$ (left), ${\littleb_1}$ (right) of $B$.\\Their exceptional vertices are labeled $e$, and they\\have multiplicity $2$.}}
        \label{fig:blocks-of-B}
    \end{figure}
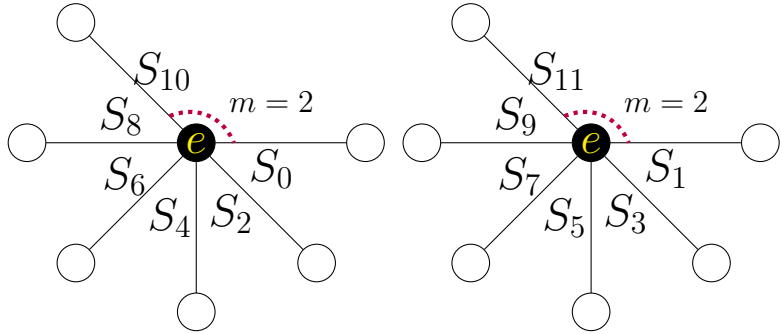
\end{prop}
\begin{proof}
    These are given in \cite[\S 10.3.2]{bonnafe}.
\end{proof}
We now want to give a description of the indecomposable modules for $G$; however, this is not as easy as the case for $B$. We describe all the indecomposables in stages.
\begin{prop}
    \label{simple F[G] modules}
    Let ${1\leq t\leq p}$. Define the vector subspace ${V_t\leq \mathbb{F}[x,y]}$ as
    $$
    {
        V_t = \angel{x^{t-1},\ x^{t-2}y,\ ...,\ xy^{t-2},\ y^t}_{\mathbb{F}}.
    }
    $$
    Then under the action
    $$
    {
        x^iy^j \cdot
        \left(
            \begin{array}{ll}
                \alpha&\beta\\
                \gamma&\delta
            \end{array}
        \right)
        :=
        (\alpha x + \beta y)^i(\gamma x + \delta y)^j,
    }
    $$
    the modules ${V_1,\ V_2,\ ...,\ V_p}$ give a complete list of the simple ${\mathbb{F}[G]}$ modules up to isomorphism.
\end{prop}
\begin{proof}
    See \cite[\S 10.1.2]{bonnafe} or \cite[pp. 14-16]{localrep}.
\end{proof}
We now give a description of the projective indecomposable modules.
\begin{prop}
    \label{projective indecomposable F[G] modules}
    Let ${1\leq i\leq p}$, and let ${P_{V_i}}$ denote the projective cover of the simple module ${V_i}$.
    \begin{itemize}
        \item ${\underline{i=1}}$. ${P_{V_1}}$ is a uniserial module of dimension $p$, with composition factors ${V_1, V_{p-2}, V_1}$ in ascending order.
        \item ${\underline{1 < i < p}}$. ${P_{V_i}}$ is of dimension ${2p}$ and has the three socle layers ${V_i, V_{p+1-i} \oplus V_{p-1-i}, V_i}$.
        \item ${\underline{i=p}}$. ${P_{V_p} \cong V_p}$ is both simple and projective.
    \end{itemize}
\end{prop}
\begin{proof}
    See \cite[\S \RN{2}.7]{localrep}.
\end{proof}
We now give the Brauer trees for the group $G$.
\begin{prop}
    \label{block structure for F[G]}
    The group algebra ${\mathbb{F}[G]}$ consists of three blocks ${\mathbb{F}[G] \cong \mathcal{B}_0 \times \mathcal{B}_1 \times \mathcal{B}_2}$. The first two blocks ${\mathcal{B}_0,\mathcal{B}_1}$ are the Brauer correspondents of ${\littleb_0,\littleb_1}$ respectively, and ${\mathcal{B}_2}$ is the semisimple block containing the only projective simple module ${V_p}$. Let ${\varepsilon \in \{-1,1\}}$ such that ${p\equiv \varepsilon\Modwb{4}}$. The Brauer trees of ${\mathcal{B}_0,\mathcal{B}_1}$ are:
    \begin{figure}[H] 
        \centering
        \input{indecomposables/tikz/brauer-tree-SL2p-principal-block}
        \input{indecomposables/tikz/brauer-tree-SL2p-nonprincipal-block}
        \caption{The Brauer trees ${\mathcal{B}_0}$ (top), ${\mathcal{B}_1}$ (bottom) of $G$.}
        \label{fig:blocks-of-G}
    \end{figure}
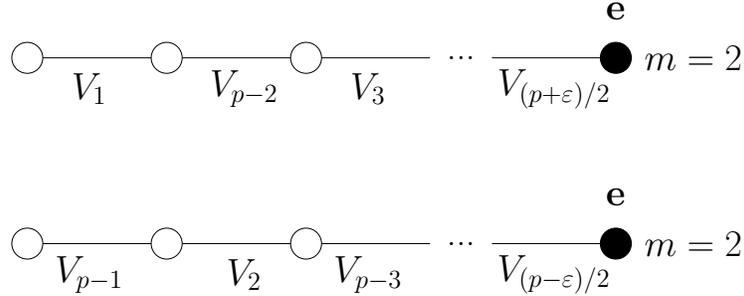
\end{prop}
\begin{proof}
    The Brauer trees for the blocks ${\mathcal{B}_0,\mathcal{B}_1}$ are given in \cite[\S \RN{5}.17]{localrep}, \cite[\S 10.3.2]{bonnafe}.
\end{proof}
We now apply results from \cite{janusz} to show how certain types of walks on the Brauer trees parameterise the non-projective indecomposables belonging to that block.
\begin{thm}
    \label{5.1}
    Given the Brauer tree of a block ${\mathcal{B}}$ with cyclic defect, the non-simple non-projective uniserial modules belonging to this block can be parameterised as follows:
    \begin{enumerate}
        \item Let ${E,E'}$ be two edges on the Brauer tree such that ${E}$ and ${E'}$ intersect at a single vertex $P$. If ${P}$ is exceptional, let $m$ be its multiplicity, or otherwise set ${m = 1}$. Then for ${n = 1,...,m}$, there exists a unique non-simple non-projective uniserial ${\mathcal{B}}$-module with socle ${E}$, top ${E'}$, and both show up $n$ times as a composition factor.
        \item Let ${E}$ be an edge of the Brauer tree incident to the exceptional vertex. Let ${n = 2,...,m}$, where $m$ is the multiplicitly of the exceptional vertex. Then there exists a unique uniserial ${\mathcal{B}}$-module with socle $E$ and top $E$, showing up $n$ times as a composition factor.
    \end{enumerate}
\end{thm}
\begin{proof}
    See \cite[\S 5]{janusz}.
\end{proof}
We will denote a module coming from Theorem \ref{5.1}(1) by ${M(E',E,n)}$, and rather naturally denote a module coming from Theorem \ref{5.1}(2) by ${M(E,E,n)}$ (exactly as done in \cite{janusz}). As we now demonstrate using $G$, Theorem \ref{5.1} essentially says that the uniserial modules belonging to a cyclic block are the uniserial submodules of the projective indecomposables.
\begin{cor}
    \label{uniserial F[G] modules}
    There are ${2(p-2)}$ non-simple non-projective uniserial ${\mathbb{F}[G]}$ modules. They all have length 2, and have the following descriptions. One has socle ${V_1}$ with top ${V_{p-2}}$. One has socle ${V_{p-1}}$ with top ${V_2}$. Finally, for ${1 < i < p-1}$, there exists such a module with socle ${V_i}$ and top ${V_{p+1-i}}$, and also such a module with socle ${V_i}$ and top ${V_{p-1-i}}$.
\end{cor}
\begin{proof}
    We first use Theorem \ref{5.1} to count how many non-simple non-projective uniserial modules belong to one of the Brauer trees of ${G}$. Doubling this number then gives us the total number of non-simple non-projective uniserial ${\mathbb{F}[G]}$ modules. The number of pairs of distinct edges ${E,E'}$ adjacent to a common vertex $P$ on one of the Brauer trees of $G$ is equal to ${2\left(\frac{p-1}{2} - 1\right) = p-3}$. Furthermore, $P$ is always non-exceptional in this case. Hence the number of instances of Theorem \ref{5.1}(1) that occur is simply ${p-3}$. The only edge incident to the exceptional vertex is the final one. Since ${m=2}$, this gives us only one instance of \ref{5.1}(2). Hence overall, there are ${2(p-3+1) = 2(p-2)}$ non-simple non-projective uniserial ${\mathbb{F}[G]}$ modules.\\
    \\
    Finally, note from Proposition \ref{projective indecomposable F[G] modules}, we have that:
    \begin{itemize}
        \item ${\rad(P_{V_1})}$ is a uniserial module with composition factors ${V_1, V_{p-2}}$ in ascending order.
        \item ${\rad(P_{V_{p-1}})}$ is a uniserial module with composition factors ${V_{p-1}, V_2}$ in ascending order.
        \item For ${1 < i < p}$, ${\rad(P_{V_i})}$ contains a submodule which is uniserial with composition factors ${V_{i}, V_{p+1-i}}$, and a submodule which is uniserial with composition factors ${V_i, V_{p-1-i}}$.
    \end{itemize}
    These match the descriptions given in the Corollary statement. By counting, we have ${2(p-2)}$ of them. Thus, we have successfully described the non-projective, non-simple uniserial ${\mathbb{F}[G]}$ modules.
\end{proof}
We now continue following \cite{janusz} and explain exactly how any non-simple non-projective indecomposable module belonging to a block with cyclic defect can be described by two types of walks on the corresponding Brauer tree, along with two extra variables. I will give the full generalised statements here, then explain how they simplify in the case of our trees in figure \ref{fig:blocks-of-G}.
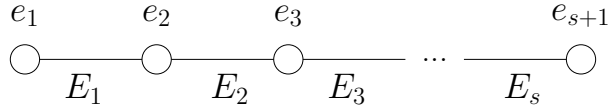
\begin{figure}[H]
    \centering
    \input{indecomposables/tikz/typeIwalk}
    \caption{Walk type \RN{1}, ${s\geq 2}$}
    \label{fig:5.2(a)}
\end{figure}
\begin{figure}[H]
    \centering
    \input{indecomposables/tikz/typeIIwalk}
    \caption{Walk type \RN{2}, ${s\geq 2}$}
    \label{fig:5.2(b)}
\end{figure}
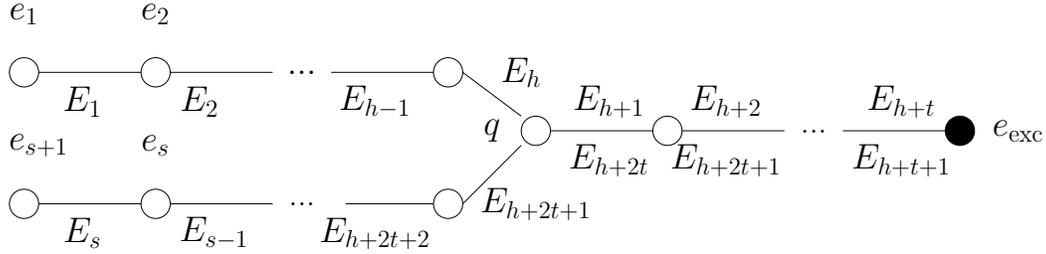
All the edges in a walk of type \RN{1} are distinct, and is simply a straight line walk. It is possible that one of the vertices is exceptional. A walk of type \RN{2} is as follows: we start the walk at a non-exceptional vertex ${e_1}$, walk up to the exceptional vertex ${e_{\text{exc}}}$, walk back the way we came by at least one edge, then potentially branch off to another set of distinct edges at some vertex $q$. It is possible to have ${e_1 = q}$, ${q = e_{s+1}}$, and even ${e_1 = q = e_{s+1}}$.\\
\\
Now we are required to make more choices. Fix a walk of type \RN{1} or \RN{2}. Let $D$ be either all the even integers in ${\{1,2,...,s\}}$, or all the odd integers. Then, for ${1\leq i\leq s-1}$, we let
\begin{equation}
    \label{M_i definition}
    M_i =
    \left\{
        \begin{array}{ll}
             M(E_{i},E_{i+1},n_{i})&\text{if }i \in D  \\
             M(E_{i+1},E_{i},n_{i})&\text{if }i\notin D 
        \end{array}
    \right.
\end{equation}
for some choices of ${n_i}$ adhering to the conditions laid out in Theorem \ref{5.1}. Note that because at most one vertex in walks \RN{1} or \RN{2} is exceptional, at most one of the ${n_i}$ is not equal to $1$. Let ${\psi_i}$ denote an inclusion of ${\soc(M_i)}$ into ${M_i}$ and ${\varphi_i}$ denote an epimorphism of ${M_i}$ onto ${\Top(M_i)}$. Finally, let
\begin{align}
    \label{X definition}
    X &= \left\{(m_1, ..., m_{s-1}) \in M_1 \oplus ... \oplus M_{s-1}: \varphi_i(m_i) = \varphi_{i+1}(m_{i+1})\text{ in }E_{i+1}\text{ for }i \in D^c,\ 1 \leq i < s-1\right\} \\
    \label{Y definition}
    Y &= \angel{(0,...,0,\psi_{i-1}(f),\psi_i(f),0,...,0) \in M_1 \oplus ... \oplus M_{s-1}: f \in E_{i}, i \in D^c,\ 1\leq i\leq s-1}\subseteq X \\
    \label{W definition}
    W &= X/Y .
\end{align}
As explained in \cite[\S \RN{5}]{janusz}, $Y$ is a submodule of $X$ because the reducibility of the ${M_i}$ implies ${\psi_i}$ maps ${\soc(M_i)}$ into the kernel of ${\varphi_i}$.
\begin{thm}
    \label{module W}
    Suppose we are given a block ${\mathcal{B}}$ of cyclic defect. Let $W$ be a module constructed as in (\ref{W definition}) for some walk ${E_1,...,E_s}$ of type \RN{1} or \RN{2}, a choice of ${D}$ and a choice of the ${n_i}$. Then we have the following:
    \begin{enumerate}
        \item The isomorphism type of $W$ is independent on the choices of the maps ${\varphi_i,\psi_i}$.
        \item We have ${\Top(W) \cong \bigoplus_{i \in D}E_i^{\oplus n_i}}$ and ${\soc(W) \cong \bigoplus_{i \in D^c}E_i^{\oplus n_i}}$.
        \item W is non-simple, non-projective and indecomposable.
    \end{enumerate}
    Furthermore, we have:
    \begin{enumerate}[label=(\alph*)]
        \item Any non-simple non-projective indecomposable module belonging to ${\mathcal{B}}$ is isomorphic to some $W$ constructed this way.
        \item A non-simple non-projective indecomposable module $M$ uniquely determines a walk of either type \RN{1} or \RN{2} up to reversal of order. Once the order is fixed, $M$ uniquely determines the multiplicities ${n_i}$ as in (\ref{M_i definition}). The set $D$ is also uniquely determined unless the walk is of type \RN{2} with ${e_1 = q = e_{s+1}}$; in this case, either choice of $D$ gives the same isomorphism class of module.
    \end{enumerate}
\end{thm}
\begin{proof}
    (1) is \cite[(5.7)]{janusz}. (2) is proved there shortly after. In regards to (3), the fact $W$ is indecomposable is \cite[(5.12)]{janusz}. The fact that $W$ is non-simple and non-projective is given in the closing comments of \cite[\S 6]{janusz}. (a) is \cite[(5.16)]{janusz}. Finally, (b) is \cite[(5.9)]{janusz}.
\end{proof}
For notational convenience, rather than using the set $D$ as defined above, we simply use the notation
\begin{equation}
    \label{epsilon definition}
    \epsilon =
    \left\{
        \begin{array}{ll}
             -1&\text{if }D = \{2,4,6,...\}  \\
             1&\text{if }D = \{1,3,5,...\} 
        \end{array}
    \right.
\end{equation}
to keep track of the two possible choices. In other words, if ${\epsilon = -1}$, then our walk is of the form ``socle, top, socle, ...". If ${\epsilon = 1}$, then our walk is of the form ``top, socle, top, ...".\\
\\
Let's now focus on applying all the above to parameterise and count the number of non-simple non-projective indecomposables in the case of our specific trees in figure \ref{fig:blocks-of-G}. The first thing to note is that the ${n_i}$ are already uniquely determined by the type of walk \RN{1} or \RN{2}.\\
\\
Next, because our trees are straight lines (of length ${(p-1)/2}$), a walk of type \RN{2} starts at some non-exceptional vertex (i.e. any vertex except the last). Then, we walk all the way right to the exceptional vertex (the last one), and then walk back by at least one edge. Because our tree is a straight line, we always have ${e_1 = q}$ or ${q = e_{s+1}}$; however, by Theorem \ref{module W}(b), by simply reversing the walk and changing the choice of $\epsilon$ as necessary, we can always assume ${e_{s+1} = q}$.  The number of walks of type \RN{2} with ${e_1 = q = e_{s+1}}$ is given by the number of edges ${(p-1)/2}$, and by Theorem \ref{module W}(b) this gives us ${(p-1)/2}$ modules. The number of walks of type \RN{2} with ${e_{s+1} = q}$ but ${e_1 \neq q}$ is given by ${\sum_{i=1}^{(p-3)/2}i = \frac{(p-3)(p-1)}{8}}$. Because of the two choices for $\epsilon$ giving two different modules (Theorem \ref{module W}(b)), we must double this number to get the number of modules that arise this way.\\
\\
Now we focus on modules associated with walks of type \RN{1}. Because a walk of type \RN{2} naturally begins by going from left to right, we will take the same convention for these walks and go from left to right. The number of such walks is equal to the number of sub-trees of our tree of length ${\geq 2}$, which is given by ${\sum_{i=1}^{(p-3)/2}i}$. Once again, we double this to get the total number of modules arising this way to account for the two choices of $\epsilon$.\\
\\
To get then the total number of non-projective indecomposable modules belonging to one of our blocks, we simply take the sum of the above counts and add in the number of simple non-projective modules, to get
$$
{
    \frac{p-1}{2} + 4\frac{(p-3)(p-1)}{8} + \frac{p-1}{2} = \frac{(p-1)^2}{2}
}
$$
Before giving our final description, it should be noted that although the walks \RN{1} and \RN{2} parameterise only the non-simple non-projective indecomposables, the simple non-projectives can be included as walks of length 1.
\begin{defn}
    \label{parameterisation of the non-projective indecomposable F[G] modules}
    Any quadruple ${(t,l,s,\epsilon) \in\{0,1\} \times \{0,1,...,(p-3)/2\} \times \{1,2,...,p-1\} \times \{-1,1\}}$ with \\${l + s \leq p-1}$ defines a walk as above, and thus a non-projective indecomposable ${\mathbb{F}[G]}$ module ${M = M(i,l,s,\epsilon)}$ as follows: 
    \begin{itemize}
        \item $M$ belongs to block ${\mathcal{B}_t}$.
        \item The walk begins at the ${(l+1)^{\thh}}$ vertex in the tree.
        \item $s$ is the length of the walk. The walk goes from left to right, changing direction if we loop around the exceptional vertex. As ${l + s \leq p-1}$, the walk does not go beyond the first vertex of the tree. If ${s=1}$, $M$ is the simple ${\mathbb{F}[G]}$ module corresponding to the edge in the walk.
        \item If ${s\geq 2}$ (so that we have a walk of type \RN{1} or \RN{2}), ${\epsilon}$ defines $D$ as in (\ref{epsilon definition}). If ${s=1}$, then ${M(i,l,s,\epsilon)}$ is independent of the choice of ${\epsilon}$.
    \end{itemize}
\end{defn}
Walks of type \RN{1} using the description in definition \ref{parameterisation of the non-projective indecomposable F[G] modules} are those with ${s\geq 2}$ and ${l + s \leq (p-1)/2}$. The walks of type \RN{2} are those with ${l + s > (p-1)/2}$. This description is not one-to-one: for example, we have not added in any restriction on ${l, s}$ so that walks of type \RN{2} always satisfy ${e_{s+1}=q}$. If desired, this can be easily added in, but for our purposes will be unnecessary. We finish this section with a final Corollary:
\begin{cor}
    \label{all non-simple non-projective indecomposable F[G] modules have Loewy length 2}
    Let $M$ be a non-simple non-projective indecomposable ${\mathbb{F}[G]}$ module. Then we have the following:
    \begin{enumerate}[label=(\alph*)]
        \item $M$ has Lowey length 2.
        \item The multiset of composition factors of $M$ is equal to the multiset ${\{E_1,...,E_s\}}$ of edges showing up in the walk determined by $M$ according to Theorem \ref{module W}(b).
    \end{enumerate}
\end{cor}
\begin{proof}
    By Theorem \ref{module W}(a), $M$ is a quotient of a submodule of a direct sum of non-simple non-projective uniserial ${\mathbb{F}[G]}$ modules. By Proposition \ref{uniserial F[G] modules}, these have Loewy length 2, and so by standard properties of Loewy length, $M$ has Loewy length ${\leq 2}$. The Lowey length of $M$ must be ${\geq 2}$ (or else they would be simple), and thus we conclude $M$ has Lowey length 2.\\
    \\
    Note that it is then an easy consequence of the fact that an indecomposable module $M$ of Loewy length $2$ satisfies ${\soc(M) = \rad(M)}$ that the multiset of composition factors of $M$ is equal to the union of the multisets of composition factors of ${\soc(M)}$ and ${\Top(M)}$. Then by Theorem \ref{module W}(2), this is equal to the multiset of edges ${\{E_1,...,E_s\}}$ showing up in the walk for $M$.
\end{proof}
We finish off this section by describing the Stable Auslander-Reiten quivers ${\Gamma_{\littleb_0}^{s},\Gamma_{\littleb_1}^s}$ of the blocks ${\littleb_0,\littleb_1}$ of ${\mathbb{F}[B]}$. Recall that the Stable A.R. quiver of a cyclic block is a finite tube (for example, this is given in \cite[\S 2, Prop. 2.1]{bessenrodt} and \cite[pp. 121]{bleher-chinburg}).
\begin{thm}
    \label{almost split exact sequences for F[B]}
    Let ${0\leq a\leq p-2}$ and ${1\leq b\leq p}$. Then there exists an almost split exact sequence of the form
    $$
    {
        0\to U_{a,b} \to U_{a,b+1} \oplus U_{a+2,b-1} \to U_{a+2,b}\to 0 .
    }
    $$
    Furthermore, this is all of them.
\end{thm}
\begin{proof}
    This is given in \cite[pp. 190]{benson}.
\end{proof}
As a Corollary of the above, we obtain the following:
\begin{cor}
    \label{stable a.r. quivers for F[B]}
    Let ${i \in \{0,1\}}$. The Stable A.R. quiver for the block ${\littleb_i}$ of ${\mathbb{F}[B]}$ is:
    \begin{figure}[H]
        \centering
        \input{greencorrespondence/tikz/stable-ar-quiver-of-b_i}
        \caption{Stable A.R. quiver of the block ${\littleb_i}$ of ${\mathbb{F}[B]}$}
    \label{fig:quiver-of-b_i}
    \end{figure}
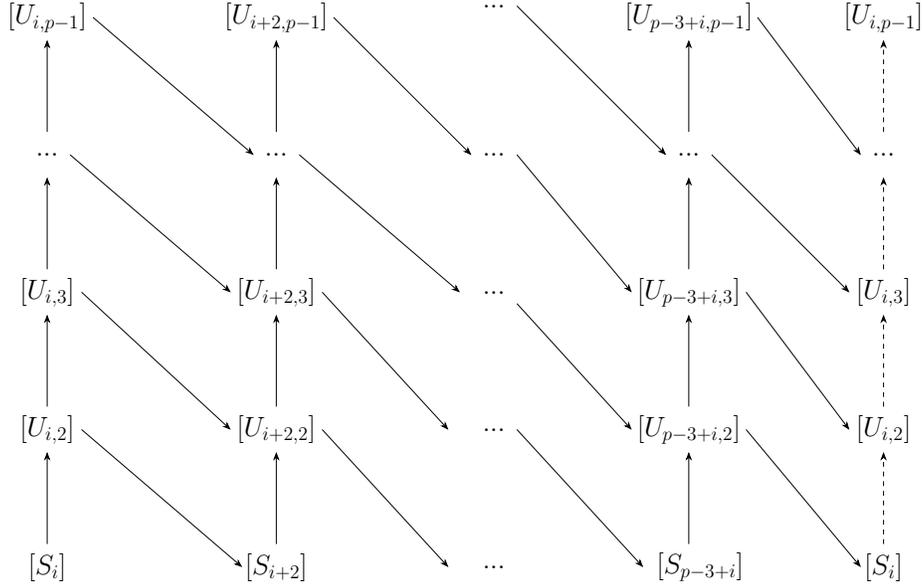
    Furthermore, ${\Omega^2(U_{a,b}) = U_{a-2,b}}$. I.e. ${\Omega^2}$ moves a module one vertex to the left.
\end{cor}
\begin{cor}
    \label{F[B] distances between hooks}
    Let ${0\leq a \leq a' \leq p-2}$ such that ${a,a'}$ are either both odd or both even. Then the length of a minimal path from ${S_a}$ to ${S_{a'}}$ on the Stable A.R. quiver is given by ${a'-a}$. The length of a minimal path from ${U_{a,p-1}}$ to ${U_{a',p-1}}$ on the Stable A.R. quiver is also given by ${a'-a}$.
\end{cor}
\begin{proof}
    This is clear from figure \ref{fig:quiver-of-b_i}.
\end{proof}

%% file: indecomposables/tikz/brauer-tree-B-b0.tex
\begin{circuitikz}
\tikzstyle{every node}=[font=\LARGE]
\draw [ fill={rgb,255:red,0; green,0; blue,0} ] (11,11.75) circle (0.25cm);
\node [font=\large] at (12,12.3) {$m=2$};
\draw [short] (11.25,11.75) -- (13,11.75);
\draw  (13.25,11.75) circle (0.25cm);

\draw [short] (11.18,11.57) -- (12.42,10.33);
\draw  (12.6,10.15) circle (0.25cm);
\draw [short] (10.82,11.57) -- (9.58,10.33);
\draw  (9.4,10.15) circle (0.25cm);
\draw [short] (10.82,11.93) -- (9.58,13.17);
\draw  (9.4,13.35) circle (0.25cm);
\draw[short, dotted, ultra thick, purple] (10.646,12.104) .. controls (11,12.25) and (11.354,12.104) .. (11.5,11.75);
\node [font=\LARGE] at (12,11.4) {$S_0$};
\node [font=\LARGE] at (11.46,10.80) {$S_2$};
\node [font=\LARGE] at (10.65,10.75) {$S_4$};
\node [font=\LARGE] at (10.05,11.30) {$S_6$};
\node [font=\LARGE] at (10,12.1) {$S_8$};
\node [font=\LARGE] at (10.54,12.70) {$S_{10}$};
\node [font=\LARGE] at (11,11.75) {\textcolor{yellow}{$e$}};

\draw [short] (11,11.5) -- (11,9.75);
\draw  (11,9.5) circle (0.25cm);
\draw [short] (10.75,11.75) -- (9,11.75);
\draw  (8.75,11.75) circle (0.25cm);
\end{circuitikz}

%% file: indecomposables/tikz/brauer-tree-B-b1.tex
\begin{circuitikz}
\tikzstyle{every node}=[font=\LARGE]
\draw [ fill={rgb,255:red,0; green,0; blue,0} ] (11,11.75) circle (0.25cm);
\node [font=\large] at (12,12.3) {$m=2$};
\draw [short] (11.25,11.75) -- (13,11.75);
\draw  (13.25,11.75) circle (0.25cm);

\draw [short] (11.18,11.57) -- (12.42,10.33);
\draw  (12.6,10.15) circle (0.25cm);
\draw [short] (10.82,11.57) -- (9.58,10.33);
\draw  (9.4,10.15) circle (0.25cm);
\draw [short] (10.82,11.93) -- (9.58,13.17);
\draw  (9.4,13.35) circle (0.25cm);
\draw[short, dotted, ultra thick, purple] (10.646,12.104) .. controls (11,12.25) and (11.354,12.104) .. (11.5,11.75);
\node [font=\LARGE] at (12,11.4) {$S_1$};
\node [font=\LARGE] at (11.46,10.80) {$S_3$};
\node [font=\LARGE] at (10.65,10.75) {$S_5$};
\node [font=\LARGE] at (10.05,11.30) {$S_7$};
\node [font=\LARGE] at (10,12.1) {$S_9$};
\node [font=\LARGE] at (10.54,12.70) {$S_{11}$};
\node [font=\LARGE] at (11,11.75) {\textcolor{yellow}{$e$}};

\draw [short] (11,11.5) -- (11,9.75);
\draw  (11,9.5) circle (0.25cm);
\draw [short] (10.75,11.75) -- (9,11.75);
\draw  (8.75,11.75) circle (0.25cm);
\end{circuitikz}

%% file: indecomposables/tikz/brauer-tree-SL2p-principal-block.tex
\begin{figure}[H]
\centering
\resizebox{0.6\textwidth}{!}{%
\begin{circuitikz}[scale = 1]
\tikzstyle{every node}=[font=\LARGE]
\draw  (2,10.5) circle (0.25cm);
\draw (2.25,10.5) to[short] (4,10.5);
\draw  (4.25,10.5) circle (0.25cm);
\draw [](4.5,10.5) to[short] (6.25,10.5);
\draw  (6.5,10.5) circle (0.25cm);
\draw [](6.75,10.5) to[short] (8.5,10.5);
\node [font=\LARGE] at (9,10.5) {...};
\draw [](9.5,10.5) to[short] (11.25,10.5);
\draw[fill = black]  (11.5,10.5) circle (0.25cm);
\node [font=\LARGE] at (11.5,11.25) {$\mathbf{e}$};
\node [font=\LARGE] at (3,10) {$V_1$};
\node [font=\LARGE] at (5.5,10) {$V_{p-2}$};
\node [font=\LARGE] at (7.5,10) {$V_3$};
\node [font=\LARGE] at (10.5,10) {$V_{(p+\varepsilon)/2}$};
\node [font=\LARGE] at (12.75,10.5) {${m=2}$};
\end{circuitikz}
}%

\end{figure}

%% file: indecomposables/tikz/brauer-tree-SL2p-nonprincipal-block.tex
\begin{figure}[H]
\centering
\resizebox{0.6\textwidth}{!}{%
\begin{circuitikz}[scale = 1]
\tikzstyle{every node}=[font=\LARGE]
\draw  (2,10.5) circle (0.25cm);
\draw (2.25,10.5) to[short] (4,10.5);
\draw  (4.25,10.5) circle (0.25cm);
\draw [](4.5,10.5) to[short] (6.25,10.5);
\draw  (6.5,10.5) circle (0.25cm);
\draw [](6.75,10.5) to[short] (8.5,10.5);
\node [font=\LARGE] at (9,10.5) {...};
\draw [](9.5,10.5) to[short] (11.25,10.5);
\draw[fill = black]  (11.5,10.5) circle (0.25cm);
\node [font=\LARGE] at (11.5,11.25) {$\mathbf{e}$};
\node [font=\LARGE] at (3,10) {$V_{p-1}$};
\node [font=\LARGE] at (5.5,10) {$V_{2}$};
\node [font=\LARGE] at (7.5,10) {$V_{p-3}$};
\node [font=\LARGE] at (10.5,10) {$V_{(p-\varepsilon)/2}$};
\node [font=\LARGE] at (12.75,10.5) {${m=2}$};
\end{circuitikz}
}%

\end{figure}

%% file: indecomposables/tikz/typeIwalk.tex
\begin{figure}[H]
\centering
\resizebox{0.5\textwidth}{!}{%
\begin{circuitikz}[scale = 1]
\tikzstyle{every node}=[font=\LARGE]
\draw  (2,10.5) circle (0.25cm);
\draw (2.25,10.5) to[short] (4,10.5);
\draw  (4.25,10.5) circle (0.25cm);
\draw [](4.5,10.5) to[short] (6.25,10.5);
\draw  (6.5,10.5) circle (0.25cm);
\draw [](6.75,10.5) to[short] (8.5,10.5);
\node [font=\LARGE] at (9,10.5) {...};
\draw [](9.5,10.5) to[short] (11.25,10.5);
\draw  (11.5,10.5) circle (0.25cm);
\node [font=\LARGE] at (2,11.25) {$e_1$};
\node [font=\LARGE] at (4.25,11.25) {$e_2$};
\node [font=\LARGE] at (6.5,11.25) {$e_3$};
\node [font=\LARGE] at (11.5,11.25) {$e_{s+1}$};
\node [font=\LARGE] at (3,10) {$E_1$};
\node [font=\LARGE] at (5.5,10) {$E_2$};
\node [font=\LARGE] at (7.5,10) {$E_3$};
\node [font=\LARGE] at (10.5,10) {$E_s$};
\end{circuitikz}
}%

\end{figure}

%% file: indecomposables/tikz/typeIIwalk.tex
\begin{figure}[H]
\centering
\resizebox{0.85\textwidth}{!}{%
\begin{circuitikz}
\tikzstyle{every node}=[font=\LARGE]
\draw  (4,12.5) circle (0.25cm);
\draw (4.25,12.5) to[short] (6,12.5);
\draw  (6.25,12.5) circle (0.25cm);
\draw [](6.5,12.5) to[short] (8.25,12.5);
\node [font=\LARGE] at (8.75,12.5) {$...$};
\draw [](9.25,12.5) to[short] (11,12.5);
\draw  (11.25,12.5) circle (0.25cm);
\draw  (12.75,11.5) circle (0.25cm);
\draw [short] (11.5,12.5) -- (12.5,11.75);
\draw  (11.25,10.25) circle (0.25cm);
\draw [short] (12.5,11.25) -- (11.5,10.25);
\draw [short] (11,10.25) -- (9.5,10.25);
\node [font=\LARGE] at (8.75,10.25) {...};
\draw[] (8.25,10.25) to[short] (6.5,10.25);
\draw  (6.25,10.25) circle (0.25cm);
\draw [short] (6,10.25) -- (4.25,10.25);
\draw  (4,10.25) circle (0.25cm);
\draw [](13,11.5) to[short] (14.75,11.5);
\draw  (15,11.5) circle (0.25cm);
\draw [](15.25,11.5) to[short] (17,11.5);
\node [font=\LARGE] at (17.5,11.5) {...};
\draw [](18,11.5) to[short] (19.75,11.5);
\draw[fill = black]  (20,11.5) circle (0.25cm);
\node [font=\LARGE] at (4,13.5) {$e_1$};
\node [font=\LARGE] at (6.25,13.5) {$e_2$};
\node [font=\LARGE] at (12.5,12.5) {$E_h$};
\node [font=\LARGE] at (12,11.5) {$q$};
\node [font=\LARGE] at (21,11.5) {$e_{\text{exc}}$};
\node [font=\LARGE] at (14,12) {$E_{h+1}$};
\node [font=\LARGE] at (16,12) {$E_{h+2}$};
\node [font=\LARGE] at (19,12) {$E_{h+t}$};
\node [font=\LARGE] at (19,11) {$E_{h+t+1}$};
\node [font=\LARGE] at (16,11) {$E_{h+2t+1}$};
\node [font=\LARGE] at (14,11) {$E_{h+2t}$};
\node [font=\LARGE] at (12.75,10.25) {$E_{h+2t+1}$};
\node [font=\LARGE] at (7.25,9.75) {$E_{s-1}$};
\node [font=\LARGE] at (5,9.75) {$E_s$};
\node [font=\LARGE] at (5,12) {$E_1$};
\node [font=\LARGE] at (7,12) {$E_2$};
\node [font=\LARGE] at (10,12) {$E_{h-1}$};
\node [font=\LARGE] at (10,9.75) {$E_{h+2t+2}$};
\node [font=\LARGE] at (6.25,11.25) {$e_{s}$};
\node [font=\LARGE] at (4.25,11.25) {$e_{s+1}$};
\end{circuitikz}
}%

\end{figure}

%% file: greencorrespondence/tikz/stable-ar-quiver-of-b_i.tex
\begin{figure}[H]
\centering
\resizebox{0.75\textwidth}{!}{%
\begin{circuitikz}
\tikzstyle{every node}=[font=\LARGE]
\node [font=\LARGE] at (6.5,8) {$[S_i]$};
\draw [->, >=Stealth] (6.5,8.5) -- (6.5,10.5);
\node [font=\LARGE] at (6.5,11) {$[U_{i,2}]$};
\node [font=\LARGE] at (6.5,14) {$[U_{i,3}]$};
\node [font=\LARGE] at (7,9) {};
\draw [->, >=Stealth] (6.5,11.5) -- (6.5,13.5);
\draw [->, >=Stealth] (6.5,14.5) -- (6.5,16.5);
\node [font=\LARGE] at (6.5,17) {$...$};
\draw [->, >=Stealth] (6.5,17.5) -- (6.5,19.5);
\node [font=\LARGE] at (6.5,20) {$[U_{i,p-1}]$};
\node [font=\LARGE] at (11.5,8) {$[S_{i+2}]$};
\draw [->, >=Stealth] (11.5,8.5) -- (11.5,10.5);
\node [font=\LARGE] at (11.5,11) {$[U_{i+2,2}]$};
\node [font=\LARGE] at (11.5,14) {$[U_{i+2,3}]$};
\node [font=\LARGE] at (12,9) {};
\draw [->, >=Stealth] (11.5,11.5) -- (11.5,13.5);
\draw [->, >=Stealth] (11.5,14.5) -- (11.5,16.5);
\node [font=\LARGE] at (11.5,17) {$...$};
\draw [->, >=Stealth] (11.5,17.5) -- (11.5,19.5);
\node [font=\LARGE] at (11.5,20) {$[U_{i+2,p-1}]$};
\node [font=\LARGE] at (16.75,9) {};
\node [font=\LARGE] at (16.25,17) {$...$};
\node [font=\LARGE] at (20.75,8) {$[S_{p-3+i}]$};
\draw [->, >=Stealth] (20.5,8.5) -- (20.5,10.5);
\node [font=\LARGE] at (20.5,11) {$[U_{p-3+i,2}]$};
\node [font=\LARGE] at (20.5,14) {$[U_{p-3+i,3}]$};
\node [font=\LARGE] at (21,9) {};
\draw [->, >=Stealth] (20.5,11.5) -- (20.5,13.5);
\draw [->, >=Stealth] (20.5,14.5) -- (20.5,16.5);
\node [font=\LARGE] at (20.5,17) {$...$};
\draw [->, >=Stealth] (20.5,17.5) -- (20.5,19.5);
\node [font=\LARGE] at (20.5,20) {$[U_{p-3+i,p-1}]$};
\draw [->, >=Stealth] (7.5,20) -- (11,17);
\node [font=\LARGE] at (24.75,8) {$[S_i]$};
\node [font=\LARGE] at (24.75,11) {$[U_{i,2}]$};
\node [font=\LARGE] at (24.75,14) {$[U_{i,3}]$};
\node [font=\LARGE] at (25.25,9) {};
\node [font=\LARGE] at (24.75,17) {$...$};
\node [font=\LARGE] at (24.75,20) {$[U_{i,p-1}]$};
\draw [->, >=Stealth] (7,17) -- (10.5,14);
\draw [->, >=Stealth] (7.25,14) -- (10.5,11);
\draw [->, >=Stealth] (7.25,11) -- (10.75,8);
\draw [->, >=Stealth] (12,17) -- (15.5,14);
\draw [->, >=Stealth] (12.5,14) -- (15.25,11);
\draw [->, >=Stealth] (12.5,11) -- (15.25,8);
\node [font=\LARGE] at (16.25,20.25) {$...$};
\node [font=\LARGE] at (16.25,14) {$...$};
\node [font=\LARGE] at (16.25,11) {$...$};
\node [font=\LARGE] at (16.25,8) {$...$};
\draw [->, >=Stealth] (16.75,20.25) -- (20,17);
\draw [->, >=Stealth] (16.75,17) -- (19.25,14);
\draw [->, >=Stealth] (16.75,13.75) -- (19.25,11);
\draw [->, >=Stealth] (16.75,11) -- (19.5,8);
\draw [->, >=Stealth] (21,17) -- (24,14);
\draw [->, >=Stealth] (21.75,14) -- (24,11);
\draw [->, >=Stealth] (21.75,11) -- (24.25,8);

\draw [->, >=Stealth, dashed] (24.75,8.5) -- (24.75,10.5);
\draw [->, >=Stealth, dashed] (24.75,11.5) -- (24.75,13.5);
\draw [->, >=Stealth, dashed] (24.75,14.5) -- (24.75,16.5);
\draw [->, >=Stealth] (22,20) -- (24.25,17);
\draw [->, >=Stealth, dashed] (24.75,17.5) -- (24.75,19.5);
\draw [->, >=Stealth] (12.75,20) -- (15.75,17);
\end{circuitikz}
}%

\end{figure}

%% file: greencorrespondence/greencorrespondence.tex
We now have a nice description of the non-projective indecomposable ${\mathbb{F}[G]}$ modules based on walks on the Brauer trees. This description is required for applying results from \cite{bleher-chinburg}. This paper allows us obtain information about the location of a module in the Stable Auslander-Reiten quiver from its corresponding Brauer tree walk. We use the fact that the Green correspondence provides an isomorphism ${\Gamma_{\mathcal{B}_i}^s \cong \Gamma_{\littleb_i}^s}$ of these quivers to explicitly describe the bijection of non-projective indecomposable modules (see \cite[\S \RN{5}.17, Thm 3]{localrep}, \cite[\S \RN{5}.17, cor. 4]{localrep}, the comments following \cite[\S \RN{1}.10, Prop. 1.4]{auslander-reiten}). Given a non-projective indecomposable ${\mathbb{F}[B]}$ module ${U_{a,b}}$, we will denote its Green correspondent by ${V_{a,b}}$.
\begin{defn}
    \label{maximal directed path definition}
    Let ${\Lambda}$ be a cyclic block of a group algebra. A path in ${\Gamma_{\Lambda}^s}$ is called \textit{directed} if it contains no sub-path from ${\Omega^2(X)}$ to $X$ for any non-projective indecompsable module $X$.\\
    A directed path is called \textit{maximal} if it is directed and ends at a boundary.
\end{defn}
\begin{rem}
    \label{maximal directed path remark}
    It's not very difficult to see that in fact, any module $M$ lying on the Stable A.R. quiver of a cyclic block will have exactly two maximally directed paths. Furthermore, since Green correspondence commutes with ${\Omega^2}$, in fact maximally directed paths are sent to maximally directed paths under the Green correspondence.
\end{rem}
\begin{cor}
    \label{maximal directed paths for F[B]}
    Let ${U_{a,b}}$ be a non-projective indecomposable ${\mathbb{F}[B]}$ module. Then there are precisely two maximally directed paths starting at ${U_{a,b}}$. The first is given by
    $$
    {
        [U_{a,b}] \to [U_{a,b+1}] \to ... \to [U_{a,p-1}]
    }
    $$
    with a length of ${p - 1 - b}$. The second is given by
    $$
    {
        [U_{a,b}] \to [U_{a+2,b-1}] \to [U_{a+4,b-2}] \to ... \to [S_{a + 2(b-1)}]
    }
    $$
    with a length of ${b-1}$.
\end{cor}
Let $M$ be a non-projective indecomposable module in a cyclic block. By the \textit{boundary modules} of $M$, we mean the two modules that the two maximal directed paths starting at $M$ end at. For example, we have seen in Corollary \ref{maximal directed paths for F[B]} that given a ${U_{a,b}}$, the two boundary modules are ${U_{a,p-1}}$ and ${S_{a + 2(b-1)}}$.
\begin{defn}
    \label{hook definition}
    Let ${H}$ be a non-projective uniserial module. Let ${W = \rad(H)}$ and ${S = \Top(H)}$. Then we call $H$ a \textit{hook} if ${\rad(P_S) / \soc(P_S) \cong W\oplus Q}$ for some uniserial module $Q$.
    \\
    Let $C$ also be a non-projective uniserial module. Let ${R = \soc(C)}$ and let ${X = C/R}$. Then $C$ is called a \textit{cohook} if ${\rad(P_R) / \soc(P_R) \cong X\oplus Q'}$ for some uniserial module ${Q'}$.
\end{defn}
\begin{rem}
    \label{hook cohook remark}
    As given in \cite[\S \RN{2}.3, pp. 114]{bleher-chinburg}, for Brauer tree algebras every hook is a cohook and vice versa. So for our purposes, we will use these notions interchangeably. It is also explained that given a Brauer tree algebra ${\Lambda}$ with $e$ isomorphism classes of simple modules, then there are ${2e}$ cohooks. Finally, cohooks are precisely the modules that lie on the boundaries of ${\Gamma^s_{\Lambda}}$. Each boundary consists of $e$ cohooks.
\end{rem}
\begin{lem}
    \label{hooks of F[G]}
    The two blocks ${\mathcal{B}_0}$ and ${\mathcal{B}_1}$ of ${\mathbb{F}[G]}$ have ${p-1}$ cohooks each. They are given by the non-simple non-projective uniserial modules as in Corollary \ref{uniserial F[G] modules}, as well as the two simple modules ${V_1}$ and ${V_{p-1}}$.
\end{lem}
\begin{proof}
    We run through the modules given in Corollary \ref{uniserial F[G] modules}. First, let $C$ be the uniserial module with socle ${V_1}$ and top ${V_{p-2} = C/V_{1}}$ (the latter equality coming from the fact the non-simple uniserials have length 2). From the Brauer tree, ${\rad(P_{V_1})/\soc(P_{V_1}) \cong 0 \oplus V_{p-2}}$, and thus $C$ is a cohook. It also follows from this calculation that the simple module ${V_1}$ is a cohook.\\
    \\
    Next, let $C$ be the uniserial module with socle ${V_{p-1}}$ and top ${V_2 = C/V_{p-1}}$. From the Brauer tree, \\${\rad(P_{V_{p-1}})/\soc(P_{V_{p-1}}) \cong 0 \oplus V_2}$, and thus $C$ is a cohook. It also follows from this calculation that the simple module ${V_{p-1}}$ is a cohook.\\
    \\
    For ${1 < i < p-1}$, let ${C}$ be the uniserial module with socle ${V_i}$ and top ${V_{p+1-i} = C/V_i}$. Then from the Brauer tree ${\rad(P_{V_i})/\soc(P_{V_i}) \cong P_{V_{p+1-i}} \oplus P_{V_{p-1-i}}}$, and thus $C$ is a cohook. This also proves the uniserial module with socle ${V_i}$ and top ${V_{p-1-i}}$ is a cohook.\\
    \\
    A simple count shows that we thus have found ${p-1}$ cohooks belonging to each block, which must then be all of them by remark \ref{hook cohook remark}.
\end{proof}
Now that we understand what the hooks of ${\mathbb{F}[G]}$ look like, we want to show how they split to form the boundaries of ${\Gamma_{\mathcal{B}_i}^s}$. We also want to give a method of computing the minimal distance between two hooks on the same boundary. For this, we recall \cite[\S \RN{3}, Prop. 3.7]{bleher-chinburg}. Note that by the clockwise walk around the Brauer tree, we mean a walk which, at each stage, takes the next edge clockwise around our current vertex to the one we just took.
\begin{prop}
    \label{distance between hooks on the same boundary}
    Let ${\Lambda}$ be a Brauer tree algebra, with $e$ isomorphism classes of simple modules. Let $m$ denote the multiplicity of the exceptional vertex. If ${e = 1}$ and ${m > 1}$, then there is exactly one vertex in each of the two boundaries of ${\Gamma_{\Lambda}^s}$. Otherwise, let $H$ be a hook. Let ${W_H = (x_1, X_1, x_2, X_2, ..., X_{2e},x_{2e+1})}$ be the unique clockwise walk on the Brauer tree with ${X_1 = \Top(H)}$ and ${X_2 = \soc(H)}$. Then a hook ${H'}$ belongs to the same boundary as $H$ if there exists a ${1 \leq j\leq e}$ with ${\Top(H') = X_{2j-1}}$ and ${\soc(H') = X_{2j}}$. The length of a minimal path from ${H}$ to ${H'}$ in ${\Gamma^s_{\Lambda}}$ is given by ${2(j-1)}$. The length of a minimal path from ${H'}$ to $H$ is given by ${2e - 2(j-1)}$.
\end{prop}
\begin{cor}
    \label{boundaries of Gamma_s(B_i)}
    Let ${i \in \{0,1\}}$. Then one boundary of ${\Gamma_{\mathcal{B}_i}^s}$ consists of all hooks with dimension ${\equiv 1\Modwb{p}}$, while the other boundary consists of all hooks with dimension ${\equiv p-1\Modwb{p}}$. Furthermore, we have the following:
    \begin{itemize}
        \item A non-simple hook belonging to ${\mathcal{B}_0}$, which is of the form ${M(0,l,2,\epsilon)}$, belongs to the same boundary as ${V_1}$ if and only if ${l}$ is odd.
        \item A non-simple hook belonging to ${\mathcal{B}_1}$, which is of the form ${M(1,l,2,\epsilon)}$, belongs to the same boundary as ${V_{p-1}}$ if and only if ${l}$ is odd.
        \item The minimal distance from ${V_1}$ (resp. ${V_{p-1}}$) to ${M(0,l,2,\epsilon)}$ (resp. ${M(1,l,2,\epsilon)}$) with ${l}$ odd is ${l+1}$ if ${\epsilon = 1}$, and ${p-2-l}$ if ${\epsilon = -1}$.
    \end{itemize}
\end{cor}
\begin{proof}
    Once again, if ${p = 3}$ so that ${e = (p-1)/2 = 1}$, then the statements become vacuously true. Let's suppose then that ${p > 3}$ so that ${e > 1}$. We will prove this result only for ${\mathcal{B}_0}$ as the proof is the same for ${\mathcal{B}_1}$. It's easy to see from Lemma \ref{hooks of F[G]} that all hooks of ${\mathcal{B}_i}$ have dimension ${\equiv 1\Modwb{p}}$ or ${\equiv p-1\Modwb{p}}$. Furthermore, from the description of the edges on the Brauer tree, it is not difficult to also see that the non-simple hooks with dimension ${\equiv 1\Modwb{p}}$ are of the form ${M(0,l,2,\epsilon)}$ for some odd ${l}$. If we apply Proposition \ref{distance between hooks on the same boundary} with ${H = V_1}$, then the walk ${W_H}$ can be described as follows (the reader may want to do an example of this to clarify this description):
    \begin{itemize}
        \item ${W_H}$ starts at the second vertex. It takes the ${V_1}$ edge to the first vertex, then the ${V_1}$ edge back to the second vertex.
        \item ${W_H}$ takes the next two edges to the right. This tells us that the hook ${M(0,1,2,1)}$ belongs to the same boundary as ${V_1}$, and the distance from ${V_1}$ to it is ${2}$.
        \item Once again, ${W_H}$ takes the next two edges to the right. This tells us that the hook ${M(0,3,2,1)}$ belongs to the same boundary as ${V_1}$, and the distance from ${V_1}$ to this hook is ${4}$.
        \item This pattern continues, adding ${2}$ to ${l}$ each time, all the way up until we meet the exceptional vertex and start walking to the left. If ${(p-1)/2}$ is odd, the last hook we get before turning around is ${M(0,(p-5)/2,2,1)}$. If ${(p-1)/2}$ is even, we actually turn around ``mid-hook": \\${M(0,(p-3)/2,2,1) = M(0,(p-3)/2,2,-1)}$.
        \item Now the walk ${W_H}$ begins walking back the way we came, accumulating the hooks with the same ${l}$ values in reverse order and with ${\epsilon = -1}$. For example, if ${(p-1)/2}$ is odd, the next hook on this walk we get is ${M(0,(p-5)/2 - 2,2,-1)}$. Then ${M(0,(p-5)/2 - 4,2,-1)}$. All the way until ${M(0,1,2,-1)}$.
        \item To see the distance from ${V_1}$ to the ${\epsilon=-1}$ hooks is given by ${p-2-l}$, note first of all that the distance from ${V_1}$ to ${M(0,1,2,-1)}$ is ${p-3 = p-2-1}$. This is the hook that has the maximal distance from ${V_1}$ to it, as it is the last hook that shows up in the ${W_H}$ walk.
        \item By playing ${W_H}$ in reverse, we see that ${M(0,3,1,-1)}$ has distance ${p-3-2 = p-2-3}$ from ${V_1}$ to it. ${M(0,5,1,-1)}$ has distance ${p-3-4 = p-2-5}$. This pattern clearly continues, and in general ${M(0,l,1,-1)}$ has distance ${p-3-(l-1) = p-2-l}$ from ${V_1}$ to it.
    \end{itemize}
\end{proof}
We can now describe some of the Green correspondence.
\begin{lem}
    \label{restriction of simple F[G] modules}
    For ${1\leq i\leq p}$ we have
    $$
    {
        \Res^G_B(V_i) = U_{p-i,i}.
    }
    $$
\end{lem}
\begin{proof}
    For ${2 \leq i\leq p}$, it follows from \cite[\S \RN{1}.3, pp. 16]{localrep} that ${\Span_{\mathbb{F}}\{y^{i-1}\} \subseteq V_i}$ is the socle of ${V_i}$ as a ${\mathbb{F}[B]}$ module. Note
    $$
    {
        y^{i-1}\cdot \lambda = y^{i-1}\zeta^{1-i} = y^{i-1}\zeta^{p - i}.
    }
    $$
    From Proposition \ref{parameterisation of F[B] modules}, this gives us ${\Res^G_B(V_i) = U_{p-i,i}}$, as required.
\end{proof}
\begin{cor}
    \label{small correspondence}
    Let ${0\leq a\leq p-2}$ be even. Then
    $$
    {
        V_{a,1}
        =
        \left\{
            \begin{array}{ll}
                 M(0,0,1,-1)&\text{if }a=0  \\
                 M(0,a-1,2,1)&\text{if }a \in [1,(p-1)/2]  \\
                 M(0,p-2-a,2,-1)&\text{if }a \in [(p+1)/2,p-2]
            \end{array}
        \right. .
    }
    $$
    Let ${0\leq a\leq p-2}$ be odd. Then
    $$
    {
        V_{a,p-1}
        =
        \left\{
            \begin{array}{ll}
                 M(1,0,1,-1)&\text{if }a=1  \\
                 M(1,a-2,2,1)&\text{if }a \in [2,(p-1)/2]  \\
                 M(1,p-1-a,2,-1)&\text{if }a \in [(p+1)/2,p-2]
            \end{array}
        \right. .
    }
    $$
\end{cor}
Note if ${a = (p+1)/2}$, then ${M(1,p-1-a,2,-1) = M(1,a-2,2,1)}$.
\begin{proof}
    We know from Lemma \ref{restriction of simple F[G] modules} that ${V_{0,1} = V_1 = M(0,0,1,-1)}$. Let ${0 < a \leq p-2}$ be even. Then, we know ${S_a}$ is on the same boundary as ${S_0}$, and the minimal distance from ${S_0}$ to ${S_a}$ is ${a}$. Note the Green correspondent ${V_{0,1}}$ of ${S_0}$ is ${V_1}$. To find the Green correspondent ${V_{a,1}}$ of ${S_{a}}$, we simply need to find the hook lying on the same boundary as ${V_1}$ such that the minimal distance from ${V_1}$ to it is ${a}$. If ${0 < a \leq (p-1)/2}$, by Corollary \ref{boundaries of Gamma_s(B_i)}, we have ${V_{a,1} = M(0,a-1,2,1)}$. If ${(p-1)/2 < a \leq p-2}$, we have ${V_{a,1} = M(0,p-2-a,2,-1)}$.\\
    \\
    We know from Lemma \ref{restriction of simple F[G] modules} that ${V_{1,p-1} = V_{p-1} = M(1,0,1,-1)}$. Now let ${1 < a \leq p-2}$ be odd. We know ${U_{a,p-1}}$ is on the same boundary as ${U_{1,p-1}}$, and the minimal distance from ${U_{1,p-1}}$ to ${U_{a,p-1}}$ is ${a-1}$. Similar to before, we get if ${a \leq (p+1)/2}$, then ${V_{a,p-1} = M(1,a-2,2,1)}$. Otherwise, ${V_{a,p-1} = M(1,p-1-a,2,-1)}$.
\end{proof}
We have been successful at describing a small portion of the Green correspondence. In particular, we have described the Green correspondence for one of the boundaries in each of the two blocks. To describe the rest of the Green correspondence for modules not lying on these boundaries, we need more machinery.
\begin{thm}
    \label{Theorem 3.5}
    Let $M$ be a non-projective indecomposable module belonging to a Brauer tree algebra.
    \begin{enumerate}[label = (\alph*)]
        \item If $M$ is simple, let ${(e_1,E_1,e_2)}$ be a walk on the Brauer tree so that ${E_1 = M}$. Let ${H,H'}$ be the two unique hooks with ${\soc(H) = \soc(H') = E_1}$, and such that ${\Top(H),\Top(H')}$ are the next counter-clockwise edges to ${E_1}$ around ${e_1,e_2}$ respectively.
        \item If $M$ is non-simple, let ${(e_1,E_1,e_2,...,E_s,e_{s+1})}$, ${\epsilon}$ be a walk on the Brauer tree for $M$ as in Theorem \ref{module W}(5), (\ref{epsilon definition}). Let ${H}$ be the unique hook with ${\soc(H) = E_1}$, and ${\Top(H)}$ being the next counter-clockwise edge to ${E_1}$ around ${e_1}$ if ${\epsilon=1}$ or ${e_2}$ if ${\epsilon=-1}$. Let ${H'}$ be defined the same way using the reverse of the walk ${(e_{s+1},E_s,...,E_2,e_1)}$, ${\delta}$, where
        \begin{equation}
            \label{delta definition}
            \delta =
            \left\{
                \begin{array}{ll}
                     \epsilon&\text{if }s\text{ odd}  \\
                     -\epsilon&\text{if }s\text{ even}
                \end{array}
            \right. .
        \end{equation}
    \end{enumerate}
    Then ${H,H'}$ are the two boundary modules for $M$.
\end{thm}
\begin{proof}
    This is a version of \cite[\S \RN{3}, Thm. 3.5]{bleher-chinburg}. The Theorem as given there, with the surrounding definitions, has been setup in such a way as to only return one of the two boundary modules. In order to get the other, we have to reapply the Theorem using the reverse of the walk. This is explained at the start of \cite[\S 3]{bleher-chinburg}, where the authors mention being able to reduce to \textit{right-oriented paths} by taking the \textit{mirror image} of a module as necessary.
\end{proof}
There is some slight nuance in Theorem \ref{Theorem 3.5} which we briefly discuss. In part ${(a)}$, ${e_1}$ or ${e_2}$ could be a leaf vertex. If ${e_1}$ (resp. ${e_2}$) is a non-exceptional leaf vertex, then ${H}$ (resp. ${H'}$) is the simple module given by ${E_1}$. If ${e_1}$ (resp. ${e_2}$) is exceptional with multiplicity $m$, then $H$ (resp. ${H'}$) is the uniserial module of length $m$ whose composition factors are all given by ${E_1}$. In part ${(b)}$, if ${\epsilon=1}$ and ${e_1}$ is a leaf vertex, then ${H}$ satisfies the same description just given. If ${\epsilon=-1}$ and ${e_2}$ is a leaf vertex (note it would need to be exceptional in this case), ${H'}$ satisfies the same description just given. Finally, it should be noted that the variable ${\delta}$ defined in (\ref{delta definition}) is used to ensure that, for example, a walk that goes ``top, socle, ..." is appropriately changed to ``socle, top, ..." under reversing the direction of the walk (as in Theorem \ref{module W}(b)). \\
\\
We now give the result of applying the above Theorem \ref{Theorem 3.5} to the specific case of our ${\mathbb{F}[G]}$ modules. Once the direction of a walk on the Brauer tree is fixed, note that Theorem \ref{Theorem 3.5} says that the boundary modules for $M$ only depend on the endpoints. As an intuitive device, we refer to the boundary module associated with the start of a walk as the \textit{left boundary}, and the boundary module associated to the end of a walk as the \textit{right boundary}.
\begin{cor}
    \label{boundaries of F[G] modules}
    Let ${M = M(i,l,s,\epsilon)}$ be a non-projective indecomposable ${\mathbb{F}[G]}$ module. Then, the two boundary modules ${H,H'}$ for $M$ are as follows:
    \begin{description}
        \item[\underline{Left boundary $H$.}] We have
        $$
        {
            H
            =
            \left\{
                \begin{array}{ll}
                     M(i,0,1,-1)&\text{if }l=0,\ s=1\text{ or }l=0,\ s>1,\ \epsilon=1  \\
                     M(i,l-1,2,1)&\text{if }l>0,\ s=1\text{ or }l>0,\ s>1,\ \epsilon=1 \\
                     M(i,l,2,-1)&\text{if }s>1,\ \epsilon=-1
                \end{array}
            \right. .
        }
        $$
        \item[\underline{Right boundary $H'$.}] Next, let ${\delta}$ be as given in (\ref{delta definition}). Then if ${l + s \leq (p-1)/2}$, we have
        $$
        {
            H'
            =
            \left\{
                \begin{array}{ll}
                     M(i,l,2,-1)&\text{if }s=1  \\
                     M(i,l+s-2,2,1)&\text{if }s>1,\ l+s \leq (p-1)/2,\ \delta = -1  \\
                     M(i,l+s-1,2,-1)&\text{if }s>1,\ l+s\leq (p-1)/2,\ \delta = 1
                \end{array}
            \right. .
        }
        $$
        Otherwise, if ${l + s > (p-1)/2}$, ${H'}$ is the same as left boundary of ${M(i,p-1-l-s,s,\delta)}$.
    \end{description}
\end{cor}
\begin{proof}
    This is a simple application of Theorem \ref{Theorem 3.5}.
\end{proof}
\begin{cor}
    \label{small correspondence extended}
    Let ${0\leq a\leq p-2}$ be even. Then
    $$
    {
        V_{a,p-1}
        =
        \left\{
            \begin{array}{ll}
                 M(0,0,2,-1)&\text{if }a=0  \\
                 M(0,a-2,2,1)&\text{if }a \in [1,(p-1)/2]  \\
                 M(0,p-1-a,2,-1)&\text{if }a \in [(p+1)/2,p-2]
            \end{array}
        \right. .
    }
    $$
\end{cor}
\begin{proof}
    We have computed ${V_{a,1}}$ for ${a}$ even in Corollary \ref{small correspondence}. We know that the boundary modules for ${U_{a,1}}$ are itself and ${U_{a,p-1}}$. Hence, ${V_{a,p-1}}$ is the other boundary module for ${V_{a,1}}$. Using Corollary \ref{boundaries of F[G] modules}, we get what is stated in the Corollary statement.
\end{proof}
For ${(l,s) \in \{0,1,...,(p-3)/2\} \times \{1,2,...,p-1\}}$ such that ${l + s \leq p-1}$, define
\begin{equation}
    \label{L_i(l,s)}
    L_i(l,s) = \left\{
            \begin{array}{ll}
                (1-i)p + (-1)^{i+1}\frac{2l+s+1}{2}&l,s\text{ odd}  \\
                (1-i)p + (-1)^{i+1}\frac{s}{2}&l,s\text{ even} \\
                ip + (-1)^i\frac{s}{2}&l\text{ odd, }s\text{ even} \\
                ip + (-1)^i\frac{2l+s+1}{2}&l\text{ even, }s\text{ odd} \\
            \end{array}
        \right. .
\end{equation}
Note that ${1\leq L_i(l,s)\leq p-1}$.
\begin{prop}
    \label{dimension of F[G] module modulo p}
    Let ${M = M(i,l,s,\epsilon)}$ be a non-projective indecomposable ${\mathbb{F}[G]}$ module, and let ${N}$ be its Green correspondent. We have that
    $$
    {
        \dim_{\mathbb{F}}N
        \equiv
        \dim_{\mathbb{F}}M
        \equiv
        L_i(l,s)
        \Modwb{p} .
    }
    $$
\end{prop}
\begin{proof}
    We first explain the left hand side of the above congruence. Note that by Green correspondence and Proposition \ref{parameterisation of F[B] modules}, there exists an ${0\leq a\leq p-2}$, a ${1\leq b\leq p-1}$, and a projective ${\mathbb{F}[B]}$ module ${Q}$ such that
    $$
    {
        N\cong U_{a,b},\ \Res^G_B(M) \cong N \oplus Q \cong U_{a,b} \oplus Q.
    }
    $$
    Note that ${\dim_\mathbb{F}Q \equiv 0 \Modwb{p}}$, and hence ${\dim_{\mathbb{F}}N \equiv b\equiv \dim_{\mathbb{F}}M\Modwb{p}}$.\\
    \\
    We now explain the right hand side of the congruence in our Proposition statement. Suppose first that ${i=0}$. Then, for ${1\leq s\leq (p-1)/2}$, by Corollary \ref{all non-simple non-projective indecomposable F[G] modules have Loewy length 2} we have
    \begin{flalign*}
        \dim_{\mathbb{F}}M(0,0,s,\epsilon)&=
            \sum_{j=1}^{s}
            \left\{
                \begin{array}{ll}
                     j&\text{if }j\text{ odd}  \\
                     p-j&\text{if }j\text{ even} 
                \end{array}
            \right.
            \equiv
            \sum_{j=1}^{s}(-1)^{j+1}j
            =
            \left\{
                \begin{array}{ll}
                     \frac{s+1}{2}&\text{if }s\text{ odd}  \\
                     -\frac{s}{2}&\text{if }s\text{ even}
                \end{array}
            \right.\\
            &\equiv
            \left\{
                \begin{array}{ll}
                     \frac{s+1}{2}&\text{if }s\text{ odd}  \\
                     p-\frac{s}{2}&\text{if }s\text{ even}
                \end{array}
            \right.
            =
            L_0(0,s).
    \end{flalign*}
    If ${1\leq l+s\leq (p-1)/2}$, then 
    $$
    {
        \dim_{\mathbb{F}}M(0,l,s,\epsilon) \equiv
        \dim_{\mathbb{F}}M(0,0,l+s,\epsilon)
        -
        \dim_{\mathbb{F}}M(0,0,l,\epsilon)
        \equiv
        L_0(0,l+s) - L_0(0,l)\Modwb{p},
    }
    $$
    which one can verify is equal to ${L_0(l,s)}$. If ${l + s > (p-1)/2}$, then in a similar fashion we obtain
    $$
    {
        \dim_{\mathbb{F}}M(0,l,s,\epsilon) \equiv L_0(l,(p-1)/2 - l) + L_0(p-1-l-s, l+s - (p-1)/2)
        \equiv L_0(l,s) \Modwb{p}.
    }
    $$
    To see the right hand congruence, denote the above left hand side by ${S(l,s)}$ and let ${0\leq j_1,j_2,j_3\leq 1}$ such that ${j_1,j_2,j_3}$ are equivalent to ${(p-1)/2,l,s}$ modulo $2$ respectively. Then the table
    \begin{figure}[H]
        \centering
        \input{greencorrespondence/Tables/dimension-cases-table}
    \end{figure}
    proves the claim. This completes the proof for ${i=0}$. If ${i=1}$, then note that for ${1\leq s\leq (p-1)/2}$,
    $$
    {
        \dim_{\mathbb{F}}M(1,0,s,\epsilon) \equiv
        \sum_{j=1}^{s}
        \left\{
            \begin{array}{ll}
                 j&\text{if }j\text{ even}  \\
                 p-j&\text{if }j\text{ odd} 
            \end{array}
        \right.
        \equiv
        -\dim_{\mathbb{F}}M(0,0,s,\epsilon)
        \equiv
        -L_0(0,s) \Modwb{p}
    }
    $$
     from which it follows that ${\dim_{\mathbb{F}}M(1,l,s,\epsilon)\equiv -L_0(l,s)\equiv L_1(l,s)\Modwb{p}}$, giving us the result in the Proposition statement.
\end{proof}
\begin{thm}
    \label{green correspondence G}
    Let ${1\leq b\leq p-1}$, ${0\leq a\leq p-2}$ and ${i \in \{0,1\}}$ such that ${a \equiv i\Modwb{2}}$. If ${a \in [0,1]}$, then
    $$
    {
        V_{a,b}
        =
        \left\{
            \begin{array}{ll}
                 M(i,0,2b + i - 1,2i-1)&\text{if }b \in \left[1,\frac{p-1}{2}\right]  \\
                 M(i,0,2(p-b) - i,2i-1)&\text{if }b \in \left[\frac{p+1}{2},p-1\right] 
            \end{array}
        \right. .
    }
    $$
    If ${a \in \left[2,\frac{p-1}{2}\right]}$, then
    $$
    {
        V_{a,b} =
        \left\{
            \begin{array}{ll}
                 M(i,a-1,2b,1)&\text{if }b \in \left[1,\frac{p-a}{2}\right]  \\
                 M(i,a-1,2(p-a-b)+1,1)&\text{if }b \in \left[\frac{p-a+1}{2},p-a\right] \\
                 M(i,2(p-b) - a - 1, 2(a+b-p)+1,-1)&\text{if }b \in \left[p-a, p - \frac{a+1}{2}\right] \\
                 M(i,a + 2(b-p), 2(p-b), 1)&\text{if }b \in \left[p - \frac{a}{2}, p-1\right]
            \end{array}
        \right. .
    }
    $$
    Finally, if ${a \in \left[\frac{p+1}{2},p-2\right]}$, then
    $$
    {
        V_{a,b} =
        \left\{
            \begin{array}{ll}
                 M(i,p-a-2b,2b,-1)&\text{if }b \in \left[1,\frac{p-a}{2}\right]  \\
                 M(i,a+2b-p-1,2(p-a-b)+1,1)&\text{if }b \in \left[\frac{p-a+1}{2},p-a\right] \\
                 M(i,p-1-a,2(a+b-p)+1,-1)&\text{if }b \in \left[p-a, p - \frac{a+1}{2}\right] \\
                 M(i,p-1-a,2(p-b),-1)&\text{if }b \in \left[p - \frac{a}{2}, p-1\right]
            \end{array}
        \right. .
    }
    $$
\end{thm}
\begin{proof}
    At the core of this proof is the following observation:
    \begin{itemize}
        \item The module ${U_{a,b}}$ is the unique ${\mathbb{F}[B]}$ module of dimension $b$ with ${U_{a,p-1}}$ as one of its boundary modules. It therefore follows from Proposition \ref{dimension of F[G] module modulo p} that ${V_{a,b}}$ is the unique non-projective indecomposable ${\mathbb{F}[G]}$ module whose dimension modulo $p$ is $b$, and who has ${V_{a,p-1}}$ as a boundary module.
    \end{itemize}
    First of all, suppose ${a=0}$. Then, ${V_{0,b}}$ is the unique ${\mathcal{B}_0}$ module with ${V_{0,p-1} = M(0,0,2,-1)}$ (cor. \ref{small correspondence extended}) as a boundary module and dimension ${b}$ modulo $p$. By using Corollary \ref{boundaries of F[G] modules}, all the modules with ${V_{0,p-1}}$ as a boundary module are of the form ${M(0,0,s,-1)}$ with ${s \in [1,p-1]}$. Therefore,
    $$
    {
        b = L_0(0,s)
        =
        \left\{
            \begin{array}{ll}
                 \frac{s+1}{2}&\text{if }s\text{ odd}  \\
                 p - \frac{s}{2}&\text{if }s\text{ even} 
            \end{array}
        \right. .
    }
    $$
    Rearranging for ${s}$ and adding the condition ${s \in [1,p-1]}$, we obtain
    $$
    {
        s
        =
        \left\{
            \begin{array}{ll}
                 2b-1&\text{if }b \in \left[1,\frac{p-1}{2}\right]  \\
                 2(p-b)&\text{if }b \in \left[\frac{p+1}{2},p-1\right] 
            \end{array}
        \right. .
    }
    $$
    Overall,
    $$
    {
        V_{0,b}
        =
        \left\{
            \begin{array}{ll}
                 M(0,0,2b-1,-1)&\text{if }b \in \left[1,\frac{p-1}{2}\right]  \\
                 M(0,0,2(p-b),-1)&\text{if }b \in \left[\frac{p+1}{2},p-1\right] 
            \end{array}
        \right. .
    }
    $$
    Now suppose that ${a=1}$. Then, ${V_{1,b}}$ is the unique ${\mathcal{B}_1}$ module with ${V_{1,p-1} = M(1,0,1,-1)}$  (cor. \ref{small correspondence}) as a boundary module and dimension $b$ modulo $p$. Using Corollary \ref{boundaries of F[G] modules}, all modules with ${V_{1,p-1}}$ as a boundary module are of the form ${M(1,0,s,1)}$ for ${s \in [1,p-1]}$. Therefore,
    $$
    {
        b = L_1(0,s)
        =
        \left\{
            \begin{array}{ll}
                 p - \frac{s+1}{2}&\text{if }s\text{ odd}  \\
                 \frac{s}{2}&\text{if }s\text{ even} 
            \end{array}
        \right. .
    }
    $$
    Rearranging for ${s}$ and adding the condition ${s \in [1,p-1]}$, we obtain
    $$
    {
        s
        =
        \left\{
            \begin{array}{ll}
                 2(p-b)-1&\text{if }b \in \left[\frac{p+1}{2},p-1\right]  \\
                 2b&\text{if }b \in \left[1,\frac{p-1}{2}\right] 
            \end{array}
        \right. .
    }
    $$
    Overall,
    $$
    {
        V_{1,b}
        =
        \left\{
            \begin{array}{ll}
                 M(1,0,2b,1)&\text{if }b \in \left[1,\frac{p-1}{2}\right]  \\
                 M(1,0,2(p-b)-1,1)&\text{if }b \in \left[\frac{p+1}{2},p-1\right] 
            \end{array}
        \right. .
    }
    $$
    For ${a \in [0,1]}$, we can write this in shorthand as
    $$
    {
        V_{a,b}
        =
        \left\{
            \begin{array}{ll}
                 M(i,0,2b + i - 1,2i-1)&\text{if }b \in \left[1,\frac{p-1}{2}\right]  \\
                 M(i,0,2(p-b) - i,2i-1)&\text{if }b \in \left[\frac{p+1}{2},p-1\right] 
            \end{array}
        \right. .
    }
    $$
    Suppose now that ${a \in \left[2,\frac{p-1}{2}\right]}$. Then ${V_{a,b}}$ is the unique ${\mathcal{B}_i}$ module with ${V_{a,p-1} = M(i,a-2,2,1)}$ (see corollaries \ref{small correspondence}, \ref{small correspondence extended}) as a boundary module and dimension $b$ modulo $p$. Using Corollary \ref{boundaries of F[G] modules}, we get ${M(i,a-s,s,\epsilon)}$, where ${\epsilon}$ is chosen so that ${\delta = -1}$ (as in (\ref{delta definition})) and ${s \in [1,a]}$ have ${V_{a,p-1} = M(i,a-2,2,1)}$ as a boundary module on the right. Therefore,
    $$
    {
        b=
        L_i(a-s,s)=
        \left\{
            \begin{array}{ll}
                 p - \frac{2(a-s) + s + 1}{2}&\text{if }s\text{ odd, }i=0  \\
                 p - \frac{2(a-s) + s + 1}{2}&\text{if }s\text{ odd, }i=1  \\
                 p - \frac{s}{2}&\text{if }s\text{ even, }i=0 \\
                 p - \frac{s}{2}&\text{if }s\text{ even, }i=1
            \end{array}
        \right. .
    }
    $$
    Note this is independent of $i$. Rearranging for ${s}$ and adding in the constraint ${s \in [1,a]}$, we obtain
    $$
    {
        s
        =
        \left\{
            \begin{array}{ll}
                 2(a+b-p)+1&\text{if }b \in \left[p-a, p - \frac{a+1}{2}\right]  \\
                 2(p-b)&\text{if }b \in \left[p - \frac{a}{2}, p - 1\right] 
            \end{array}
        \right. .
    }
    $$
    In order to preserve ${\delta = -1}$, we require ${\epsilon = 1}$ if ${s}$ is even and ${\epsilon = -1}$ if $s$ is odd. Next, using Corollary \ref{boundaries of F[G] modules} once again, note modules of the form ${M(i,a-1,s,1)}$ for ${s \in [1,p-a]}$ have ${V_{a,p-1} = M(i,a-2,2,1)}$ as a boundary module on the left. Therefore,
    $$
    {
        b = L_i(a-1,s)
        =
        \left\{
            \begin{array}{ll}
                 p - \frac{2(a-1) + s + 1}{2}&\text{if }s\text{ odd, }i=0  \\
                 p - \frac{2(a-1) + s + 1}{2}&\text{if }s\text{ odd, }i=1  \\
                 \frac{s}{2}&\text{if }s\text{ even, }i=0 \\
                 \frac{s}{2}&\text{if }s\text{ even, }i=1
            \end{array}
        \right. .
    }
    $$
    Once again, this is independent of $i$. Rearranging for ${s}$ and adding the constraint that ${s \in [1,p-a]}$, we obtain
    $$
    {
        s
        =
        \left\{
            \begin{array}{ll}
                 2(p-a-b) + 1&\text{if }b \in \left[\frac{p-a+1}{2},p-a\right] \\
                 2b&\text{if }b \in \left[1,\frac{p-a}{2}\right]
            \end{array}
        \right. .
    }
    $$
    Overall, for ${a \in \left[2,\frac{p-1}{2}\right]}$ we obtain
    $$
    {
        V_{a,b} =
        \left\{
            \begin{array}{ll}
                 M(i,a-1,2b,1)&\text{if }b \in \left[1,\frac{p-a}{2}\right]  \\
                 M(i,a-1,2(p-a-b)+1,1)&\text{if }b \in \left[\frac{p-a+1}{2},p-a\right] \\
                 M(i,2(p-b) - a - 1, 2(a+b-p)+1,-1)&\text{if }b \in \left[p-a, p - \frac{a+1}{2}\right] \\
                 M(i,a + 2(b-p), 2(p-b), 1)&\text{if }b \in \left[p - \frac{a}{2}, p-1\right]
            \end{array}
        \right. .
    }
    $$
    Now we consider ${a \in \left[\frac{p+1}{2},p-2\right]}$. In this case, ${V_{a,b}}$ is the unique ${\mathcal{B}_i}$ module with ${V_{a,p-1} = M(i,p-1-a,2,-1)}$ (see corollaries \ref{small correspondence}, \ref{small correspondence extended}) as a boundary module and dimension ${b}$ modulo $p$. By Corollary \ref{boundaries of F[G] modules}, note that modules of the form ${M(i,p-a-s,s,\epsilon)}$, where ${s \in [1,p-a]}$ and ${\epsilon}$ is chosen so that ${\delta = 1}$ have ${M(i,p-1-a,2,-1)}$ as a right boundary module. Therefore,
    $$
    {
        b = L_i(p-a-s,s)
        =
        \left\{
            \begin{array}{ll}
                 \frac{2(p-a-s) + s + 1}{2}&\text{if }s\text{ odd, }i=0  \\
                 \frac{2(p-a-s) + s + 1}{2}&\text{if }s\text{ odd, }i=1  \\
                 \frac{s}{2}&\text{if }s\text{ even, }i=0 \\
                 \frac{s}{2}&\text{if }s\text{ even, }i=1
            \end{array}
        \right. .
    }
    $$
    Once again, this is independent of $i$. Rearranging for ${s}$ and adding the constraint that ${s \in [1,p-a]}$, we obtain
    $$
    {
        s
        =
        \left\{
            \begin{array}{ll}
                 2(p-a-b) + 1&\text{if }b \in \left[\frac{p-a+1}{2},p-a\right] \\
                 2b&\text{if }b \in \left[1,\frac{p-a}{2}\right]
            \end{array}
        \right. .
    }
    $$
    To preserve ${\delta = 1}$, we need ${\epsilon = 1}$ if ${s}$ odd, ${\epsilon = -1}$ if $s$ even. Now we cycle through the walks for which ${M(i,p-1-a,2,-1)}$ is a left boundary. Next, note by Corollary \ref{boundaries of F[G] modules} modules of the form
    ${M(i,p-1-a,s,-1)}$ for ${s \in [1,a]}$ also have ${M(i,p-1-a,2,-1)}$ as a a boundary module on the left. Therefore,
    $$
    {
        b = L_i(p-1-a,s)
        =
        \left\{
            \begin{array}{ll}
                 \frac{2(p-1-a) + s + 1}{2}&\text{if }s\text{ odd, }i=0  \\
                 \frac{2(p-1-a) + s + 1}{2}&\text{if }s\text{ odd, }i=1  \\
                 p - \frac{s}{2}&\text{if }s\text{ even, }i=0 \\
                 p-\frac{s}{2}&\text{if }s\text{ even, }i=1
            \end{array}
        \right. .
    }
    $$
    Once again, this is independent of $i$. Rearranging for ${s}$ and adding the constraint that ${s \in [1,a]}$, we obtain
    $$
    {
        s
        =
        \left\{
            \begin{array}{ll}
                 2(a + b - p ) + 1&\text{if }b \in \left[p-a,p - \frac{a+1}{2}\right] \\
                 2(p-b)&\text{if }b \in \left[p-\frac{a}{2},p-1\right]
            \end{array}
        \right. .
    }
    $$
    Finally, overall for ${a \in \left[\frac{p+1}{2},p-2\right]}$ we obtain
    $$
    {
        V_{a,b} =
        \left\{
            \begin{array}{ll}
                 M(i,p-a-2b,2b,-1)&\text{if }b \in \left[1,\frac{p-a}{2}\right]  \\
                 M(i,a+2b-p-1,2(p-a-b)+1,1)&\text{if }b \in \left[\frac{p-a+1}{2},p-a\right] \\
                 M(i,p-1-a,2(a+b-p)+1,-1)&\text{if }b \in \left[p-a, p - \frac{a+1}{2}\right] \\
                 M(i,p-1-a,2(p-b),-1)&\text{if }b \in \left[p - \frac{a}{2}, p-1\right]
            \end{array}
        \right. .
    }
    $$
\end{proof}
To finish fully describing the Green correspondence bijection, we now take a non-projective ${\mathbb{F}[G]}$ module ${M(i,l,s,\epsilon)}$ and compute its Green correspondent.
\begin{prop}
    \label{green correspondent of M(i,l,s,epsilon)}
    Let ${M = M(i,l,s,\epsilon)}$ be a non-projective indecomposable ${\mathbb{F}[G]}$ module. If ${s=1}$, then the Green correspondent of $M$ is given by ${U_{p-L_{i}(l,s),L_i(l,s)}}$ if ${M\neq V_1}$ and ${U_{0,1}}$ if ${M=V_1}$. If ${s\geq 2}$, define the following:
    \begin{itemize}
        \item \texttt{iflag}: this is simply ${i}$, which is $0$ or $1$.
        \item \texttt{lflag}: this is $0$ when ${l=0}$, ${1}$ when ${l}$ is odd, and ${2}$ when ${l>0}$ even.
        \item \texttt{sflag}: this is $0$ when ${s}$ is even, ${1}$ when ${s}$ is odd.
        \item \texttt{eflag}: this is just ${\epsilon}$.
        \item \texttt{sumflag}: this is $0$ when ${l+s < (p-1)/2}$, $1$ when ${(p-1)/2 < l+s < p-1}$, and $2$ when ${l+s=p-1}$.
    \end{itemize}
    Then the Green correspondent of $M$ is given by ${U_{a,L_i(l,s)}}$, where $a$ can be found on the following table:
    \begin{figure}[H]
        \input{greencorrespondence/Tables/a-values}
    \end{figure}
    where `*' means any value can be taken.
\end{prop}
\begin{proof}
    Proposition \ref{dimension of F[G] module modulo p} tells us that the Green correspondent of ${M(i,l,s,\epsilon)}$ has dimension ${L_i(l,s)}$. The case of ${s=1}$ is given by Proposition \ref{restriction of simple F[G] modules}. The remaining cases for ${s > 1}$ can be enumerated using Corollaries \ref{boundaries of F[G] modules}, \ref{small correspondence}, \ref{small correspondence extended} and the variables (\texttt{iflag}, \texttt{lflag}, \texttt{sflag}, \texttt{eflag}, \texttt{sumflag}). Due to the large number of cases (of which there are 60), these were enumerated and simplified using symbolic programming. The code can be found at: \cite{symbolic-proof}.
\end{proof}
We finish this section with two Corollaries.
 \begin{cor}
    \label{Ind_B^G(S_a) composition factors}
    Let ${0\leq a\leq p-2}$. The composition factors of ${\Ind_B^G(S_a)}$ are ${V_{a+1},V_{p-a}}$. 
\end{cor}
\begin{proof}
    Let ${i \in [0,1]}$ such that ${i\equiv a\Modwb{2}}$. Provided that ${a\neq 0}$, it is easily verified using Theorem \ref{green correspondence G} that the Green correspondent ${V_{a,1}}$ of ${S_a = U_{a,1}}$ is given by
    $$
    {
        V_{a,1}
        =
        \left\{
            \begin{array}{ll}
                 M(1,0,2,1)&\text{if }a=1 \\
                 M(i,a-1,2,1)&\text{if }a \in \left[2,\frac{p-1}{2}\right] \\
                 M(i,p-a-2,2,-1)&\text{if }a \in \left[\frac{p+1}{2},p-2\right]
            \end{array}
        \right. .
    }
    $$
    In each of these cases, by looking at the Brauer tree, we find ${V_{a,1}}$ is the unique uniserial module of length $2$ with socle ${V_{a+1}}$ and top ${V_{p-a}}$. Note that this also tells us ${V_{a,1}}$ has dimension ${p+1}$, which is the same as the dimension of ${\Ind^G_B(S_a)}$. Since ${V_{a,1}}$ occurs as a summand of ${\Ind^G_B(S_a)}$, we have ${\Ind^G_B(S_a) = V_{a,1}}$ and we are done in this case.\\
    \\
    If ${a=0}$, we cannot use this method. This is because ${V_{0,1} = V_1}$ has dimension $1$, and hence ${\Ind^G_B(S_0) \cong V_1 \oplus P}$ for some projective $P$ of dimension $p$. We will indirectly show that ${P\cong V_p}$. Using \cite[\S \RN{3}.5, cor. 4]{localrep}, Proposition \ref{parameterisation of F[B] modules}, we obtain
    $$
    {
        \mathbb{F}[B] \cong \bigoplus_{a=0}^{p-2}U_{a,p}\text{ as }\mathbb{F}[B]\text{ modules, }\mathbb{F}[G] \cong \bigoplus_{i=1}^{p}P_{V_i}^{\oplus i}\text{ as }\mathbb{F}[G]\text{ modules.}
    }
    $$
    By inducing, we obtain
    $$
    {
        \bigoplus_{a=0}^{p-2}\Ind_B^G(U_{a,p}) \cong \Ind^G_B(\mathbb{F}[B]) \cong \mathbb{F}[G] \cong \bigoplus_{i=1}^{p}P_{V_i}^{\oplus i}.
    }
    $$
    Using Proposition \ref{parameterisation of F[B] modules}, it's not difficult to see that for ${0\leq a\leq p-2}$, the composition factors of ${U_{a,p}}$ are as follows: ${S_a}$ occurs $3$ times, and ${S_d}$ for ${0\leq d\leq p-2}$, ${d\neq a}$, ${d\equiv a\Modwb{2}}$ occurs twice. Passing to the Grothendieck group, we obtain
    $$
    {
        \sum_{a=0}^{p-2}\left(2\frac{p-3}{2} + 3\right)\left[\Ind^G_B(S_a)\right]
        =
        p
        \sum_{a=0}^{p-2}\left[\Ind^G_B(S_a)\right]
        =
        \sum_{i=1}^{p}i\left[P_{V_i}\right].
    }
    $$
    Note that ${V_p}$ does not occur as a composition factor of ${\Ind^G_B(S_a)}$ for ${1\leq a\leq p-2}$, but that ${[V_p]}$ occurs $p$ times on the right-hand side of the above sum. Hence, we are forced to conclude ${V_p}$ occurs as a composition factor of ${\Ind^G_B(S_0)}$, and we are done.
\end{proof}
For ${0\leq a\leq p-2}$, ${1\leq b\leq p-1}$ and ${1\leq t\leq p-1}$, let ${c_{a,b,t}}$ be the multiplicity of ${V_t}$ as a composition factor of ${V_{a,b}}$. Define
\begin{equation}
    \label{I(l,s,t) definition}
    \mathcal{I}(l,s,t)
    =
    \bm{1}_{[l+1,l+s]}(t)
    =
    \left\{
        \begin{array}{ll}
             1&\text{if }t \in [l+1,l+s]  \\
             0&\text{otherwise} 
        \end{array}
    \right.
\end{equation}
and
\begin{equation}
    \label{Ibar(l,j) definition}
    \overline{\mathcal{I}}(l,t)
    =
    \mathcal{I}(l,p-1-l,t)
    =
    \bm{1}_{[l+1,p-1]}(t)
    =
    \left\{
        \begin{array}{ll}
             1&\text{if }t \in [l+1,p-1]  \\
             0&\text{otherwise} 
        \end{array}
    \right. .
\end{equation}
Next, let ${(i,l,s,\epsilon)}$ be such that ${V_{a,b} = M(i,l,s,\epsilon)}$ as in Theorem \ref{green correspondence G}. Then we define
\begin{equation}
    \label{c_{a,b}(j) definition}
    c_{a,b}(t) =
    \mathcal{I}(l,s,t)
    +
    \overline{\mathcal{I}}(p-1-l-s,t)
    =
    \mathcal{I}(l,s,t)
    +
    \mathcal{I}(p-1-l-s,l+s,t).
\end{equation}
\begin{cor}
    \label{c_{a,b}(t) explicit description and c_{a,b,t}}
    Let ${1\leq b\leq p-1}$. If ${a \in [0,1]}$, then
    $$
    {
        \begin{array}{ll}
             c_{a,b}(t)=&\ \ 
             \left\{
             \begin{array}{lll}
                  1&\text{if }b \in \left[1,\frac{p-1}{2}\right],&t \in \left[1,2b+a-1\right]  \\
                  1&\text{if }b \in \left[\frac{p+1}{2},p-1\right],&t \in \left[1,2(p-b)-a\right]\\
                  0&\text{otherwise}
             \end{array}
             \right.
             \\
             &\\
             &+\left\{
             \begin{array}{lll}
                  1&\text{if }b \in \left[1,\frac{p-1}{2}\right],&t \in \left[p-a-2b+1,p-1\right]  \\
                  1&\text{if }b \in \left[\frac{p+1}{2},p-1\right],&t \in \left[a+2b-p,p-1\right]\\
                  0&\text{otherwise}
             \end{array}
             \right. .
        \end{array} 
    }
    $$
    Next, if ${a \in \left[2,\frac{p-1}{2}\right]}$, then
    $$
    {
        \begin{array}{ll}
             c_{a,b}(t)=&\ \ 
             \left\{
             \begin{array}{lll}
                  1&\text{if }b \in \left[1,\frac{p-a}{2}\right],&t \in \left[a,a-1+2b\right] \\
                  1&\text{if }b \in \left[\frac{p-a+1}{2},p-a\right],&t \in \left[a,2(p-b)-a\right] \\
                  1&\text{if }b \in \left[p-a,p-\frac{a+1}{2}\right],&t \in \left[2(p-b)-a,a\right]  \\
                  1&\text{if }b \in \left[p-\frac{a}{2},p-1\right],&t \in \left[2(b-p)+a+1,a\right] \\
                  0&\text{otherwise}
             \end{array}
             \right.
             \\
             &\\
             &+\left\{
             \begin{array}{lll}
                  1&\text{if }b \in \left[1,\frac{p-a}{2}\right],&t \in \left[p-a-2b+1,p-1\right]  \\
                  1&\text{if }b \in \left[\frac{p-a+1}{2},p-a\right],&t \in \left[a+2b-p,p-1\right] \\
                  1&\text{if }b \in \left[p-a,p-\frac{a+1}{2}\right],&t \in \left[p-a,p-1\right]  \\
                  1&\text{if }b \in \left[p-\frac{a}{2},p-1\right],&t \in \left[p-a,p-1\right] \\
                  0&\text{otherwise}
             \end{array}
             \right. .
        \end{array} 
    }
    $$
    If ${a \in \left[\frac{p+1}{2},p-2\right]}$, then we have
    $$
    {
        \begin{array}{ll}
             c_{a,b}(t)=&\ \ 
             \left\{
             \begin{array}{lll}
                  1&\text{if }b \in \left[1,\frac{p-a}{2}\right],&t \in \left[p-a-2b+1,p-a\right]  \\
                  1&\text{if }b \in \left[\frac{p-a+1}{2},p-a\right],&t \in \left[a+2b-p,p-a\right] \\
                  1&\text{if }b \in \left[p-a,p-\frac{a+1}{2}\right],&t \in \left[p-a,a+2b-p\right]  \\
                  1&\text{if }b \in \left[p-\frac{a}{2},p-1\right],&t \in \left[p-a,2(p-b)-a-1+p\right] \\
                  0&\text{otherwise}
             \end{array}
             \right.
             \\
             &\\
             &+\left\{
             \begin{array}{lll}
                  1&\text{if }b \in \left[1,\frac{p-a}{2}\right],&t \in \left[a,p-1\right]  \\
                  1&\text{if }b \in \left[\frac{p-a+1}{2},p-a\right],&t \in \left[a,p-1\right] \\
                  1&\text{if }b \in \left[p-a,p-\frac{a+1}{2}\right],&t \in \left[2(p-b)-a,p-1\right]  \\
                  1&\text{if }b \in \left[p-\frac{a}{2},p-1\right],&t \in \left[2(b-p)+a+1,p-1\right] \\
                  0&\text{otherwise}
             \end{array}
             \right. .
        \end{array} 
    }
    $$
    Finally, the constants ${c_{a,b,t}}$ are given as follows:
    $$
    {
        c_{a,b,t}
        =
        \left\{
            \begin{array}{ll}
                 0&\text{if }t \equiv a\Modwb{2}\\
                 c_{a,b}(t)&\text{if }t\not\equiv a\Modwb{2},\ t \in \left[1,\frac{p-1}{2}\right]  \\
                 c_{a,b}(p-t)&\text{if }t\not\equiv a\Modwb{2},\ t \in \left[\frac{p+1}{2},p-1\right] 
            \end{array}
        \right. .
    }
    $$
\end{cor}
\begin{proof}
    Firstly, to see that the ${c_{a,b,t}}$  are given as in the final statement of Corollary \ref{c_{a,b}(t) explicit description and c_{a,b,t}}, first note that if ${t\equiv a\Modwb{2}}$ then ${V_t}$ does not belong to the same block as ${V_{a,b}}$. Hence it cannot show up as a composition factor of ${V_{a,b}}$. Otherwise, suppose ${a,t}$ are different modulo $2$, and suppose ${t \in \left[1,\frac{p-1}{2}\right]}$. Then, ${V_t}$ is represented by the ${t^{\text{th}}}$ edge in the Brauer tree for the block containing ${V_{a,b}}$. To count how many times ${V_t}$ shows up as a composition factor of ${V_{a,b}}$, we simply count how many times it shows up in the walk for ${V_{a,b}}$. It's not very difficult to see that, if ${V_{a,b} = M(i,l,s,\epsilon)}$, then ${\mathcal{I}(l,s,t)}$ counts whether or not we pass the edge for ${V_t}$ for the first portion of the walk, when we are walking towards the exceptional vertex. On the other hand, ${\overline{\mathcal{I}}(p-1-l-s,t)}$ counts whether or not we pass it after looping around the exceptional vertex (i.e. for walks of type (5.2)(b)). If the walk in fact doesn't loop around the exceptional vertex, ${\overline{\mathcal{I}}(p-1-l-s,t)}$ will always return $0$. If ${t \in \left[\frac{p+1}{2},p-1\right]}$, then ${V_t}$ is instead represented by the ${(p-t)^{\text{th}}}$ edge on the Brauer tree for the block containing ${V_{a,b}}$, and the same reasoning holds.\\
    \\
    Suppose ${a \in [0,1]}$, then using Theorem \ref{green correspondence G},
    $$
    {
        \begin{array}{ll}
             c_{a,b}(t)=&\ \ \ 
             \left\{
             \begin{array}{ll}
                  \mathcal{I}(0,2b+a-1,t)&\text{if }b \in \left[1,\frac{p-1}{2}\right]  \\
                  \mathcal{I}(0,2(p-b)-a,t)&\text{if }b \in \left[\frac{p+1}{2},p-1\right] 
             \end{array}
             \right.
             \\
             &\\
             &+\left\{
             \begin{array}{ll}
                  \overline{\mathcal{I}}(p-a-2b,t)&\text{if }b \in \left[1,\frac{p-1}{2}\right]  \\
                  \overline{\mathcal{I}}(a+2b-p-1,t)&\text{if }b \in \left[\frac{p+1}{2},p-1\right] 
             \end{array}
             \right. .
        \end{array} 
    }
    $$
    which can be rewritten as given in the Corollary statement. Next, suppose that ${a \in \left[2,\frac{p-1}{2}\right]}$. Then we can write
    $$
    {
        \begin{array}{ll}
             c_{a,b}(t)=&\ \ 
             \left\{
             \begin{array}{ll}
                  \mathcal{I}(a-1,2b,t)&\text{if }b \in \left[1,\frac{p-a}{2}\right]  \\
                  \mathcal{I}(a-1,2(p-a-b)+1,t)&\text{if }b \in \left[\frac{p-a+1}{2},p-a\right] \\
                  \mathcal{I}(2(p-b)-a-1,2(a+b-p)+1,t)&\text{if }b \in \left[p-a,p-\frac{a+1}{2}\right]  \\
                  \mathcal{I}(a+2(b-p),2(p-b),t)&\text{if }b \in \left[p-\frac{a}{2},p-1\right]
             \end{array}
             \right.
             \\
             &\\
             &+\left\{
             \begin{array}{ll}
                  \overline{\mathcal{I}}(p-a-2b,t)&\text{if }b \in \left[1,\frac{p-a}{2}\right]  \\
                  \overline{\mathcal{I}}(a+2b-p-1,t)&\text{if }b \in \left[\frac{p-a+1}{2},p-a\right] \\
                  \overline{\mathcal{I}}(p-1-a,t)&\text{if }b \in \left[p-a,p-\frac{a+1}{2}\right]  \\
                  \overline{\mathcal{I}}(p-1-a,t)&\text{if }b \in \left[p-\frac{a}{2},p-1\right]
             \end{array}
             \right. .
        \end{array} 
    }
    $$
    which can be rewritten as given in the Corollary statement. Finally, suppose ${a \in \left[\frac{p+1}{2},p-2\right]}$. Then,
    $$
    {
        \begin{array}{ll}
             c_{a,b}(t)=&\ \ 
             \left\{
             \begin{array}{ll}
                  \mathcal{I}(p-a-2b,2b,t)&\text{if }b \in \left[1,\frac{p-a}{2}\right]  \\
                  \mathcal{I}(a+2b-p-1,2(p-a-b)+1,t)&\text{if }b \in \left[\frac{p-a+1}{2},p-a\right] \\
                  \mathcal{I}(p-1-a,2(a+b-p)+1,t)&\text{if }b \in \left[p-a,p-\frac{a+1}{2}\right]  \\
                  \mathcal{I}(p-1-a,2(p-b),t)&\text{if }b \in \left[p-\frac{a}{2},p-1\right]
             \end{array}
             \right.
             \\
             &\\
             &+\left\{
             \begin{array}{ll}
                  \overline{\mathcal{I}}(a-1,t)&\text{if }b \in \left[1,\frac{p-a}{2}\right]  \\
                  \overline{\mathcal{I}}(a-1,t)&\text{if }b \in \left[\frac{p-a+1}{2},p-a\right] \\
                  \overline{\mathcal{I}}(2(p-b)-a-1,t)&\text{if }b \in \left[p-a,p-\frac{a+1}{2}\right]  \\
                  \overline{\mathcal{I}}(2(b-p)+a,t)&\text{if }b \in \left[p-\frac{a}{2},p-1\right]
             \end{array}
             \right. .
        \end{array} 
    }
    $$
    which can be rewritten as given in the Corollary statement.
\end{proof}

%% file: greencorrespondence/Tables/dimension-cases-table.tex
\begin{table}[H]
\centering
\begin{tabular}{|l|l|l|l|l|l|l|l|l|}
\hline
            ${(j_1,j_2,j_3)}$ & (0,0,0)           & (0,0,1)                 & (0,1,0)           & (0,1,1)                     & (1,0,0)           & (1,1,0)         & (1,0,1)                & (1,1,1)                \\ \hline
$S(l,s)$     & $\frac{4p-s}{2}$  & $p + l + \frac{s+1}{2}$ & $p + \frac{s}{2}$ & $2p - \frac{2l + s + 1}{2}$ & $p - \frac{s}{2}$ & $\frac{s}{2}$   & $\frac{2l + s + 1}{2}$ & $p - \frac{2l+s+1}{2}$ \\ \hline
${L_0(l,s)}$ & $p - \frac{s}{2}$ & $\frac{2l+s+1}{2}$      & $\frac{s}{2}$     & $p - \frac{2l+s+1}{2}$      & $p - \frac{s}{2}$ & $\frac{s}{2}$   & $\frac{2l + s + 1}{2}$ & $p - \frac{2l+s+1}{2}$ \\ \hline
\end{tabular}
\end{table}

%% file: greencorrespondence/Tables/a-values.tex
\begin{table}[H]
\begin{tabular}{|c|c|c|}
\hline
(\texttt{iflag}, \texttt{lflag}, \texttt{sflag}, \texttt{eflag}, \texttt{sumflag})                                                                               & ${V_{a,p-1}}$ boundary          & $a$          \\ \hline
(0, 0, *, -1, 0), (0, 0, *, -1, 1), (0, 0, 0, *, 2)                                                                                                        & ${M(0, 0, 2, -1)}$              & $0$          \\ \hline
(0, 0, 0, 1, 0), (1, 0, 1, -1, 0)                                                                                                                                & ${M(i, s - 2, 2, 1)}$           & $s$          \\ \hline
(0, 0, 0, 1, 1), (1, 0, 1, -1, 1)                                                                                                                                & ${M(i, p - s - 1, 2, -1)}$      & $s$          \\ \hline
(0, 0, 1, 1, 0), (1, 0, 0, -1, 0)                                                                                                                                & $M(i, s - 1, 2, -1)$            & $p - s$      \\ \hline
(0, 0, 1, 1, 1), (1, 0, 0, -1, 1)                                                                                                                                & ${M(i, p - s - 2, 2, 1)}$       & $p - s$      \\ \hline
(0, 1, 0, -1, 0), (0, 2, 1, 1, 0), (1, 1, 1, 1, 0), (1, 2, 0, -1, 0)                                                                                             & ${M(i, l + s - 1, 2, -1)}$      & $-l + p - s$ \\ \hline
(0, 1, 0, -1, 1), (0, 2, 1, 1, 1), (1, 1, 1, 1, 1), (1, 2, 0, -1, 1)                                                                                             & ${M(i, -l + p - s - 2, 2, 1)}$  & $-l + p - s$ \\ \hline
\begin{tabular}[c]{@{}c@{}}(0, 1, *, 1, 0), (0, 1, *, 1, 1), (0, 1, 1, 1, 2), (1, 2, *, 1, 0), \\ (1, 2, *, 1, 1), (1, 2, 0, 1, 2)\end{tabular}          & ${M(i, l - 1, 2, 1)}$           & $l + 1$      \\ \hline
(0, 1, 1, -1, 0), (0, 2, 0, 1, 0), (1, 1, 0, 1, 0), (1, 2, 1, -1, 0)                                                                                             & ${M(i, l + s - 2, 2, 1)}$       & $l + s$      \\ \hline
(0, 1, 1, -1, 1), (0, 2, 0, 1, 1), (1, 1, 0, 1, 1), (1, 2, 1, -1, 1)                                                                                             & ${M(i, -l + p - s - 1, 2, -1)}$ & $l + s$      \\ \hline
\begin{tabular}[c]{@{}c@{}}(*, 1, 1, -1, 2), (0, 2, *, -1, 0), (0, 2, *, -1, 1), (0, 2, 0, *, 2), \\ (1, 1, *, -1, 0), (1, 1, *, -1, 1)\end{tabular} & ${M(i, l, 2, -1)}$              & $-l + p - 1$ \\ \hline
(1, 0, 0, *, 2), (1, 0, *, 1, 0), (1, 0, *, 1, 1), (1, 1, 1, 1, 2), (1, 2, 0, -1, 2)                                                                       & ${M(1, 0, 1, -1)}$              & $1$          \\ \hline
\end{tabular}
\end{table}

%% file: induction-and-restriction/induction-and-restriction.tex
In this final section, we explore some applications of our results:
\begin{enumerate}
    \item We take an ${\mathbb{F}[G]}$ module $M$ such that (1) the decomposition of ${\Res^G_B(M)}$ as a direct sum of indecomposable ${\mathbb{F}[B]}$ modules is known and (2) the composition factors of $M$ as an ${\mathbb{F}[G]}$ module are known, and then derive the decomposition of $M$ into indecomposable ${\mathbb{F}[G]}$ modules (see Theorem \ref{lifting decomposition}).
    \item We take an ${\mathbb{F}[B]}$ module $N$ with a known decomposition as a direct sum of indecomposable ${\mathbb{F}[B]}$ modules, and compute the decomposition of ${\Ind^G_B(N)}$ as a direct sum of indecomposable ${\mathbb{F}[G]}$ modules (see Corollary \ref{ind^G_B(U_{a,b})}). Similarly, we take an ${\mathbb{F}[G]}$ module $M$ with known decomposition as a direct sum of indecomposable ${\mathbb{F}[G]}$ modules, and compute the decomposition of ${\Res^G_B(M)}$ as a direct sum of indecomposable ${\mathbb{F}[B]}$ modules (see Lemma \ref{Res^G_B(P_{V_t})} and Theorem \ref{restriction of non-projective indecomposable F[G] module}).
\end{enumerate}
Note that the decomposition of a projective module is uniquely determined by (and can be computed from) its composition factors. This can be done by inverting the Cartan matrix for the block the projective module belongs to. This is because the columns of the Cartan matrix correspondence to the projective indecomposable modules, and the rows correspond to their composition factors. We now work this out in detail: first for $G$, and then for $B$.\\
\\
By using the Brauer tree (\ref{fig:blocks-of-G}), we see that the Cartan matrix of ${\mathcal{B}_0}$, which I will denote by ${\mathscr{B}}$, is the ${(p-1)/2\times (p-1)/2}$ matrix with the following description when ${p\geq 5}$:
\begin{itemize}
    \item The first column is given by ${(2\ 1\ 0\ ...\ 0)^T}$.
    \item The last column is given by ${(0\ 0\ ...\ 1\ 3)^T}$.
    \item For ${1 < i < (p-1)/2}$, the ${i^{\thh}}$ column is given by:
    $$
    {
        (
            \underbrace{0\ 0\ ...\ 0}_{i-2\text{ zeroes}}
            \ 1\ 2\ 1\ 
            0\ ...\ 0
        )^T.
    }
    $$
\end{itemize}
Written out, it looks like:
\[
    \begin{blockarray}{cccccc}
    P_{V_1} & P_{V_{p-2}} & P_{V_{3}} & ... & P_{V_{(p+\varepsilon)/2}} \\
    \begin{block}{(ccccc)c}
    2 & 1 & 0 & ... & 0 & V_1\\
    1 & 2 & 1 & ... & 0 & V_{p-2}\\
    0 & 1 & 2 & ... & 0 & V_3\\
    ... & ... & ... & ... & ... & ...\\
    0 & 0 & 0 & ... & 3 & V_{(p+\varepsilon)/2}\\
    \end{block}
    \end{blockarray}
\]
If ${p = 3}$, the Cartan matrix is simply given by ${(3)}$. Since the Brauer tree for ${\mathcal{B}_0}$ and ${\mathcal{B}_1}$ is the same up to relabelling, ${\mathcal{B}_1}$ will have the exact same Cartan matrix ${\mathscr{B}}$ with the rows and columns simply relabelled accordingly. For ${1\leq i,j\leq (p-1)/2}$, define
\begin{equation}
    \label{gamma_i,j}
    \Gamma_{i,j} =
    \left\{
        \begin{array}{ll}
             (-1)^{i+j}\left(i - \frac{2ij}{p}\right)&\text{if }i\leq j  \\
             (-1)^{i+j}\left(j - \frac{2ij}{p}\right)&\text{if }i > j 
        \end{array}
    \right. .
\end{equation}
\begin{lem}
    \label{inverse of the cartan matrix}
    The inverse of the Cartan matrix ${\mathscr{B}}$ for ${\mathcal{B}_0}$ (and hence also for ${\mathcal{B}_1}$) is the ${(p-1)/2 \times (p-1)/2}$ matrix ${\Gamma}$ with the entries as given in equation (\ref{gamma_i,j}).
\end{lem}
\begin{proof}
    If ${p=3}$, it's easy to verify ${\Gamma = (1/3) = (3)^{-1} = \mathscr{B}^{-1}}$. Henceforth assume ${p\geq 5}$. We will first show that ${(\mathscr{B}\Gamma)_{ii} = 1}$. For the entry in the top-left,
    $$
    {
        (\mathscr{B}\Gamma)_{1,1} =
        2\Gamma_{1,1} + \Gamma_{2,1}
        =
        2\left(1 - \frac{2}{p}\right) -
        \left(1 - \frac{4}{p}\right) = 1.
    }
    $$
    For the entry in the bottom-right,
    \begin{align*}
        (\mathscr{B}\Gamma)_{(p-1)/2,(p-1)/2} =&
        \Gamma_{(p-3)/2,(p-1)/2} + 3\Gamma_{(p-1)/2,(p-1)/2}\\
        =&(-1)^{p-2}\left(\frac{p-3}{2} - \frac{(p-3)(p-1)}{2p}\right)
        -3(-1)^{p-2}\left(\frac{p-1}{2} - \frac{(p-1)^2}{2p}\right)\\
        =&1.
    \end{align*}
    Finally, assume ${2\leq i\leq (p-3)/2}$. Then,
    \begin{align*}
        (\mathscr{B}\Gamma)_{ii} =&
        \Gamma_{i-1,i} + 2\Gamma_{i,i} + \Gamma_{i+1,i}\\
        =&-\left(i-1 - \frac{2i(i-1)}{p}\right)
        + 2\left(i - \frac{2i^2}{p}\right)
        - \left(i - \frac{2i(i+1)}{p}\right)\\
        =&1.
    \end{align*}
    Now we show that ${(\mathscr{B}\Gamma)_{ij} = 0}$ for ${j > i}$. Well, for ${i=1}$ and ${j > i=1}$,
    $$
    {
        (\mathscr{B}\Gamma)_{ij} =
        2\Gamma_{1,j} + \Gamma_{2,j}
        =
        2(-1)^{j+1}\left(1 - \frac{2j}{p}\right)
        - (-1)^{j+1}\left(2 - \frac{4j}{p}\right) = 0.
    }
    $$
    Now, for ${2\leq i\leq (p-3)/2}$ and ${j > i}$, we have
    \begin{align*}
        (\mathscr{B}\Gamma)_{ij} =&
        \Gamma_{i-1,j} + 2\Gamma_{i,j} + \Gamma_{i+1,j}\\
        =&(-1)^{i+j-1}\left(i-1 - \frac{2(i-1)j}{p}\right)
        - 2(-1)^{i+j-1}\left(i - \frac{2ij}{p}\right)
        + (-1)^{i+j-1}\left(i+1 - \frac{2(i+1)j}{p}\right)\\
        =&0.
    \end{align*}
    Since both ${\mathscr{B},\Gamma}$ are symmetric matrices, so is the product ${\mathscr{B}\Gamma}$, allowing us to conclude ${\mathscr{B}\Gamma = I_{(p-1)/2}}$. Hence, ${\Gamma = \mathscr{B}^{-1}}$.
\end{proof}
\begin{prop}
    \label{decomposition of P}
    Let $P$ be a projective ${\mathbb{F}[G]}$ module, and for ${1\leq t\leq p}$ let ${\alpha_t \in \mathbb{N}}$ denote the multiplicity of ${V_t}$ as a composition factor of $P$. We then have the following decomposition of $P$:
    $$
    {
        P \cong \bigoplus_{i=1}^{p}P_{V_i}^{\oplus n_i},
    }
    $$
    where ${n_p = \alpha_p}$, and ${n_i}$ for ${i \in \left[1,p-1\right]}$ is given by
    \begin{equation*}
        n_i
        =
        \begin{dcases}
            \sum_{1\leq j\leq (p-1)/2\text{ odd}}\Gamma_{i,j}\alpha_j +
            \sum_{1\leq j\leq (p-1)/2\text{ even}}\Gamma_{i,j}\alpha_{p-j}&\text{if }i \in \left[1,\frac{p-1}{2}\right]\text{ odd}\\
            \sum_{1\leq j\leq (p-1)/2\text{ odd}}\Gamma_{p-i,j}\alpha_j +
            \sum_{1\leq j\leq (p-1)/2\text{ even}}\Gamma_{p-i,j}\alpha_{p-j}&\text{if }i \in \left[\frac{p+1}{2},p-1\right]\text{ odd}\\
            \sum_{1\leq j\leq (p-1)/2\text{ odd}}\Gamma_{i,j}\alpha_{p-j} +
            \sum_{1\leq j\leq (p-1)/2\text{ even}}\Gamma_{i,j}\alpha_{j}&\text{if }i \in \left[1,\frac{p-1}{2}\right]\text{ even}\\
            \sum_{1\leq j\leq (p-1)/2\text{ odd}}\Gamma_{p-i,j}\alpha_{p-j} +
            \sum_{1\leq j\leq (p-1)/2\text{ even}}\Gamma_{p-i,j}\alpha_{j}&\text{if }i \in \left[\frac{p+1}{2},p-1\right]\text{ even}\\
        \end{dcases}
    \end{equation*}
\end{prop}
\begin{proof}
    Since ${V_p}$ belongs to a semi-simple block by itself, ${n_p = \alpha_p}$. Otherwise, let
    $$
    {
        \mathbf{\alpha_0} = \left(
            \begin{array}{lllll}
            \alpha_1&\alpha_{p-2}&\alpha_3&...&\alpha_{(p+\varepsilon)/2}
            \end{array}
        \right)^T,\ 
        \mathbf{\alpha_1} = \left(
            \begin{array}{lllll}
            \alpha_{p-1}&\alpha_{2}&\alpha_{p-3}&...&\alpha_{(p-\varepsilon)/2}
            \end{array}
        \right)^T.
    }
    $$
    Then by Lemma \ref{inverse of the cartan matrix},
    $$
    {
        \Gamma \mathbf{\alpha_0}
        =
        \left(
            \begin{array}{lllll}
            n_1&n_{p-2}&n_3&...&n_{(p+\varepsilon)/2}
            \end{array}
        \right)^T,\ 
        \Gamma \mathbf{\alpha_1}
        =
        \left(
            \begin{array}{lllll}
            n_{p-1}&n_{2}&n_{p-3}&...&n_{(p-\varepsilon)/2}
            \end{array}
        \right)^T,
    }
    $$
    which, after some case taking, gives the Theorem statement.
\end{proof}
\begin{thm}
    \label{lifting decomposition}
    Let $M$ be an ${\mathbb{F}[G]}$ module. For ${1\leq t\leq p}$ let ${\ell_t}$ denote the multiplicity of ${V_t}$ as a composition factor of ${M}$. Let ${\Res^G_B(M) \cong \bigoplus_{a=0}^{p-2}\bigoplus_{b=1}^{p}U_{a,b}^{\oplus n_{a,b}}}$ for some multiplicities ${n_{a,b} \in \mathbb{N}}$. Finally, let ${c_{a,b,t}}$ denote the multiplicity of ${V_t}$ as a composition factor of ${V_{a,b}}$ (which can be computed using Corollary \ref{c_{a,b}(t) explicit description and c_{a,b,t}}). Then:
    $$
    {
        M \cong \bigoplus_{a=0}^{p-2}\bigoplus_{b=1}^{p-1}V_{a,b}^{\oplus n_{a,b}} \oplus \bigoplus_{i=1}^{p}P_{V_i}^{\oplus n_i},
    }
    $$
    where ${n_i}$ is computed as in Proposition \ref{decomposition of P} with input
    $$
    {
        \alpha_t =
        \begin{dcases}
            \ell_t - \sum_{a=0}^{p-2}\sum_{b=1}^{p-1}n_{a,b}c_{a,b,t}&\text{if }t \in [1,p-1]\\
            \ell_p&\text{otherwise}
        \end{dcases}.
    }
    $$
\end{thm}
\begin{proof}
    The following proof uses the same methodology seen in \cite[\S 6]{bleher_justfordifferentials}, replacing Brauer characters with composition factors. Write ${M \cong W \oplus P}$, where $W$ is the largest non-projective ${\mathbb{F}[G]}$ summand of $M$ and $P$ is the largest projective ${\mathbb{F}[G]}$ summand of $M$. It follows immediately from the Green correspondence that
    $$
    {
        W \cong \bigoplus_{a=0}^{p-2}\bigoplus_{b=1}^{p-1}V_{a,b}^{\oplus n_{a,b}}.
    }
    $$
    To decompose $P$, observe that for ${1\leq t\leq p}$, the multiplicity ${\alpha_t}$ of ${V_t}$ as a composition factor of $P$ can be obtained by taking the multiplicity of ${V_t}$ in $M$ minus the multiplicity of ${V_t}$ in $W$. The latter is precisely as stated in the Corollary statement, and now Proposition \ref{decomposition of P} implies Theorem \ref{lifting decomposition}.
\end{proof}
\begin{lem}
    \label{composition factors of ind^G_B(U_{a,b})}
    Let ${0\leq a\leq p-2}$, let ${1\leq b\leq p}$ and let ${1\leq t\leq p}$. Denote by ${\theta_{a,b,c}}$ the multiplicity of ${S_c}$ as a composition factor of ${U_{a,b}}$, and denote by ${\ell_{a,b,t}}$ the multiplicity of ${V_t}$ as a composition factor of ${\Ind^G_B(U_{a,b})}$. Then we have
    \begin{align*}
        \theta_{a,b,c} =&
        \begin{dcases}
            0&\text{if }a\not\equiv c\Modwb{2}\\
            \floor*{\frac{2(b-1)+a-c}{p-1}} - \ceil*{\frac{a-c}{p-1}}+1&\text{otherwise}
        \end{dcases}
    \end{align*}
    and
    $$
    {
        \ell_{a,b,t}
        =
        \theta_{a,b,t-1} + \theta_{a,b,p-t}.
    }
    $$
\end{lem}
\begin{proof}
    Firstly, we have
    \begin{align*}
        \theta_{a,b,c} =&\ \#\left\{j \in \mathbb{N}: a + 2j\equiv c\Mod{p-1},\ 0\leq j\leq b-1\right\}&\text{by Proposition \ref{parameterisation of F[B] modules}}\\
        =&\ \#\left\{k \in \mathbb{Z}: \frac{c-a + k(p-1)}{2} \in \mathbb{N},\ 0\leq  \frac{c-a + k(p-1)}{2} \leq b-1\right\}\\
        =&
        \begin{dcases}
            0&\text{if }a\not\equiv c\Modwb{2}\\
            \text{Number of integers in }\left[\frac{a-c}{p-1},\frac{2(b-1)+a-c}{p-1}\right]&\text{otherwise}
        \end{dcases}\\
        =&\begin{dcases}
            0&\text{if }a\not\equiv c\Modwb{2}\\
            \floor*{\frac{2(b-1)+a-c}{p-1}} - \ceil*{\frac{a-c}{p-1}}+1&\text{otherwise}
        \end{dcases}.
    \end{align*}
    Next, passing to the Grothendieck group, we have
    \begin{align*}
        \left[
            \Ind^G_B(U_{a,b})
        \right]
        =
        &
        \sum_{c=0}^{p-2}\theta_{a,b,c}\left[\Ind^G_B(S_{c})\right]
        &\\
        =&\sum_{c=0}^{p-2}\theta_{a,b,c}\left(\left[V_{c+1}\right] + \left[V_{p-c}\right]\right)&\text{by Corollary \ref{Ind_B^G(S_a) composition factors}}
    \end{align*}
    from which the result follows.
\end{proof}
\begin{cor}
    \label{ind^G_B(U_{a,b})}
    Let ${0\leq a\leq p-2}$, ${1\leq b\leq p}$, and for ${1\leq i\leq p}$ define ${n_i}$ as in Proposition \ref{decomposition of P} using the input
    $$
    {
        \alpha_t
        =
        \begin{dcases}
            \ell_{a,b,t} - c_{a,b,t}&\text{if }b\leq p-1\\
            \ell_{a,b,t}&\text{if }b=p
        \end{dcases}
    }
    $$
    where ${\ell_{a,b,t}}$, ${c_{a,b,t}}$ are given as in Lemma \ref{composition factors of ind^G_B(U_{a,b})} and Corollary \ref{c_{a,b}(t) explicit description and c_{a,b,t}} respectively. Then we have:
    $$
    {
        \Ind^G_B(U_{a,b})
        \cong
        \left.
        \begin{dcases}
            V_{a,b}&\text{if }b \leq p-1\\
            0&\text{if }b=p
        \end{dcases}
        \right\}
        \oplus
        \bigoplus_{i=1}^{p} P_{V_i}^{\oplus n_i}.
    }
    $$
\end{cor}
\begin{proof}
    By similar reasoning used in Theorem \ref{lifting decomposition}, ${\alpha_t}$ gives the multiplicity of ${V_t}$ as a composition factor of the largest projective summand of ${\Ind^G_B(U_{a,b})}$. The result then follows from Proposition \ref{decomposition of P}.
\end{proof}
This finishes our computations for the induction from $B$ to $G$. We now turn to restriction from $G$ to $B$. To this end, as above, as prove the following analogue of Proposition \ref{decomposition of P} for $B$:
\begin{prop}
    \label{decomposition of Q}
    Let $Q$ be a projective ${\mathbb{F}[B]}$ module, and for ${0\leq c\leq p-2}$ let ${\kappa_c \in \mathbb{N}}$ denote the multiplicity of ${S_c}$ as a composition factor of $Q$. For ${0\leq r,s\leq (p-1)/2}$, let
    \begin{equation}
        \label{delta_(rs)}
        \delta_{rs} = \begin{dcases}
            (p-2)/p&\text{if }r=s\\
            -2/p&\text{otherwise}
        \end{dcases}.
    \end{equation}
    Then,
    $$
    {
        Q \cong \bigoplus_{c=0}^{p-2}U_{c,p}^{\oplus n_{c,p}},
    }
    $$
    where
    $$
    {
        n_{c,p}
        =
        \begin{dcases}
            0&\text{if }c \not\equiv i \Modwb{2}\\
            \sum_{j=1}^{(p-1)/2} \delta_{(c+2-i)/2,j}\cdot \kappa_{i+2(j-1)}&\text{otherwise}
        \end{dcases}.
    }
    $$
\end{prop}
\begin{proof}
    By using Proposition \ref{parameterisation of F[B] modules}, note that the Cartan matrix ${\bm{\gamma}}$ of ${\littleb_i}$ is the ${(p-1)/2 \times (p-1)/2}$ matrix given by
    \begin{equation}
        \label{gamma}
        \begin{blockarray}{cccccc}
        U_{i,p} & U_{i+2,p} & U_{i+4,p} & ... & U_{p-3+i,p} \\
        \begin{block}{(ccccc)c}
        3 & 2 & 2 & ... & 2 & S_i\\
        2 & 3 & 2 & ... & 2 & S_{i+2}\\
        2 & 2 & 3 & ... & 2 & S_{i+4}\\
        ... & ... & ... & ... & ... & ...\\
        2 & 2 & 2 & ... & 3 & S_{p-3+i}\\
        \end{block}
        \end{blockarray}.
    \end{equation}
    It is then not hard to show that the inverse ${\mathbf{\delta} = \mathbf{\gamma}^{-1}}$ of the above Cartan matrix is the ${(p-1)/2\times (p-1)/2}$ matrix with entries given as in (\ref{delta_(rs)}). Finally, to compute ${n_{c,p}}$, note that
    $$
    {
        \bm{\gamma}^{-1}
        \left(
            \begin{array}{lllll}
            \kappa_{i}&\kappa_{i+2}&\kappa_{i+4}&...&\kappa_{p-3+i}
            \end{array}
        \right)^T
        =
        \left(
            \begin{array}{lllll}
            n_{i,p}&n_{i+2,p}&n_{i+4,p}&...&n_{p-3+i,p}
            \end{array}
        \right)^T.
    }
    $$
\end{proof}
\begin{lem}
    \label{Res^G_B(P_{V_t})}
    Let ${1\leq t\leq p}$. Then,
    $$
    {
        \Res^G_B(P_{V_t})
        =
        \begin{dcases}
            U_{0,p}&\text{if }t \in \{1,p\}\\
            U_{t-1,p} \oplus U_{p-t,p}&\text{otherwise}
        \end{dcases}.
    }
    $$
\end{lem}
\begin{proof}
    We make use of the fact that, by Lemma \ref{composition factors of ind^G_B(U_{a,b})}, for ${0\leq a\leq p-2}$,
    \begin{equation*}
        \theta_{a,p,c}
        =
        \begin{dcases}
            0&\text{if }c\not\equiv a\Modwb{2}\\
            2&\text{if }c\equiv a\Modwb{2},\ c\neq a\\
            3&\text{if }c=a
        \end{dcases}.
    \end{equation*}
    From Proposition \ref{projective indecomposable F[G] modules}, we have
    $$
    {
        \left[P_{V_1}\right]
        =
        2\left[V_1\right] + \left[V_{p-2}\right].
    }
    $$
    Hence, using Lemma \ref{restriction of simple F[G] modules},
    $$
    {
        \left[\Res^G_B(P_{V_1})\right]
        =
        2\left[S_0\right] + \left[U_{2,p-2}\right].
    }
    $$
    Note that by Lemma \ref{composition factors of ind^G_B(U_{a,b})}, ${\theta_{2,p-2,0} = 1}$. Overall, we have that ${\Res^G_B(P_{V_1}})$ is a projective ${\mathbb{F}[B]}$ module of dimension $p$ with ${S_0}$ as a composition factor of multiplicity $3$. By Proposition \ref{parameterisation of F[B] modules} and (\ref{gamma}), we obtain ${\Res^G_B(P_{V_1}) = U_{0,p}}$. Next, note that by Proposition \ref{projective indecomposable F[G] modules} we have ${P_{V_p} = V_p}$ (since ${V_p}$ is projective), and hence by Lemma \ref{restriction of simple F[G] modules} we have ${\Res^G_B(P_{V_p}) = U_{0,p}}$. Finally, for ${1 < t < p}$, we have
    $$
    {
        [P_{V_t}] = 2[V_t] + [V_{p+1-t}] + [V_{p-1-t}].
    }
    $$
    Hence,
    $$
    {
        [\Res^G_B(P_{V_t})]
        =
        2[U_{p-t,t}] + [U_{t-1,p+1-t}] + [U_{t+1,p-1-t}].
    }
    $$
    Now we do some counting, using Lemma \ref{composition factors of ind^G_B(U_{a,b})}:
    \begin{align*}
        \theta_{p-t,t,t-1}
        =
        \begin{dcases}
            1&\text{if }t \in \left[2,\frac{p-1}{2}\right]\\
            2&\text{if }t \in \left[\frac{p+1}{2},p-1\right]
        \end{dcases},\quad
        \theta_{t-1,p+1-t,t-1}
        =
        \begin{dcases}
            2&\text{if }t \in \left[2,\frac{p+1}{2}\right]\\
            1&\text{otherwise}
        \end{dcases},\\
        \theta_{t+1,p-1-t,t-1}
        =
        \begin{dcases}
            1&\text{if }t \in \left[2,\frac{p-1}{2}\right]\\
            0&\text{otherwise}
        \end{dcases}.
        \qquad \qquad \qquad \qquad
    \end{align*}
    
    Hence, the multiplicity of ${S_{t-1}}$ as a composition factor of ${\Res^G_B(P_{V_t})}$ is
    $$
    {
        \theta_{p-t,t,t-1} + \theta_{t-1,p+1-t,t-1} + \theta_{t+1,p-1-t,t-1}=
        \begin{dcases}
            5&\text{if }t \in \left[2,\frac{p-1}{2}\right]\\
            6&\text{if }t = \frac{p+1}{2}\\
            5&\text{if }t \in \left[\frac{p+3}{2},p-1\right]
        \end{dcases}.
    }
    $$
    Therefore if ${t = (p+1)/2}$, then ${\Res^G_B(P_{V_t})}$ is a projective ${\mathbb{F}[B]}$ module of dimension ${2p}$ with ${S_{t-1}}$ as a composition factor of multiplicity ${6}$. By Proposition \ref{parameterisation of F[B] modules} and (\ref{gamma}), we obtain ${\Res^G_B(P_{V_t}) = U_{t-1,p}^{\oplus 2}}$ (which matches the Lemma statement since ${t-1 = p-t}$). If ${t\neq (p+1)/2}$, by using similar reasoning, this simply tells us that ${U_{t-1,p}}$ must be one summand of ${\Res^G_B(P_{V_t})}$. We continue counting, using Lemma \ref{composition factors of ind^G_B(U_{a,b})}:
    \begin{align*}
        \theta_{p-t,t,p-t}
        =
        \begin{dcases}
            1&\text{if }t \in \left[2,\frac{p-1}{2}\right]\\
            2&\text{if }t \in \left[\frac{p+1}{2},p-1\right]
        \end{dcases},\quad
        \theta_{t-1,p+1-t,p-t}
        =
        \begin{dcases}
            2&\text{if }t \in \left[2,\frac{p+1}{2}\right]\\
            1&\text{if }t \in \left[\frac{p+3}{2},p-1\right]
        \end{dcases},\\
        \theta_{t+1,p-1-t,p-t}
        =
        \begin{dcases}
            1&\text{if }t \in \left[2,\frac{p-1}{2}\right]\\
            0&\text{otherwise}
        \end{dcases}. \qquad \qquad \qquad \qquad
    \end{align*}
    Hence, the multiplicity of ${S_{p-t}}$ as a composition factor of ${\Res^G_B(P_{V_t})}$ is
    $$
    {
        \theta_{p-t,t,p-t} + \theta_{t-1,p+1-t,p-t} + \theta_{t+1,p-1-t,p-t}=
        \begin{dcases}
            5&\text{if }t \in \left[2,\frac{p-1}{2}\right]\\
            6&\text{if }t = \frac{p+1}{2}\\
            5&\text{if }t \in \left[\frac{p+3}{2},p-1\right]
        \end{dcases}.
    }
    $$
    If ${t = (p+1)/2}$, as expected we get ${\Res^G_B(P_{V_t}) = U_{p-t,p}^{\oplus 2} = U_{t-1,p}^{\oplus 2}}$. Otherwise, we get that ${U_{p-t,p}}$ also occurs as a summand of ${\Res^G_B(V_t)}$, and we are done.
\end{proof}
For a non-projective indecomposable ${\mathbb{F}[G]}$ module ${M(i,l,s,\epsilon)}$ as defined in Definition \ref{parameterisation of the non-projective indecomposable F[G] modules}, and for ${1\leq t\leq p}$ let ${\Theta_{i,l,s,t}}$ denote how many times ${V_t}$ occurs as a composition factor of ${M(i,l,s,\epsilon)}$. As explained in the proof of Corollary \ref{c_{a,b}(t) explicit description and c_{a,b,t}}, we know this is given by
$$
{
    \Theta_{i,l,s,t}
    =
    \begin{dcases}
        0&\text{if }t\equiv i\Modwb{2}\\
        \bm{1}_{[l+1,l+s]}(t) + \bm{1}_{[l+1,p-1]}(t)&\text{if }t\not\equiv i\Modwb{2}\text{ and }1 \leq t\leq p-1\\
        0&\text{if }t=p
    \end{dcases}.
}
$$
\begin{thm}
    \label{restriction of non-projective indecomposable F[G] module}
    Let ${M = M(i,l,s,\epsilon)}$ be a non-projective indecomposable ${\mathbb{F}[G]}$ module. Let ${a}$ be as given in Proposition \ref{green correspondent of M(i,l,s,epsilon)}, and for ${0\leq c\leq p-2}$ let ${n_{c,p}}$ be computed as in Proposition \ref{decomposition of Q} with the input
    $$
    {
        \kappa_c
        =
        \Theta_{i,l,s,1}\theta_{0,1,c}
        +
        \sum_{t=2}^{p-1}\Theta_{i,l,s,t}\theta_{p-t,t,c}
        -
        \theta_{a,L_i(l,s),c}.
    }
    $$
    Then,
    $$
    {
        \Res^G_B(M) \cong U_{a,L_i(l,s)} \oplus \bigoplus_{c=0}^{p-2}U_{c,p}^{\oplus n_{c,p}}.
    }
    $$
\end{thm}
\begin{proof}
    This proof is similar to Theorem \ref{lifting decomposition} and Corollary \ref{ind^G_B(U_{a,b})}. Write ${\Res^G_B(M) \cong N \oplus Q}$, where ${N,Q}$ are the largest non-projective ${\mathbb{F}[B]}$ summand, and largest projective ${\mathbb{F}[B]}$ summand of ${\Res^G_B(M)}$ respectively. It follows immediately from the Green correspondence that ${N\cong U_{a,L_i(l,s)}}$. Note that by Lemma \ref{restriction of simple F[G] modules}, for ${0\leq c\leq p-2}$,
    $$
    {
        \mu_c
        :=
        \Theta_{i,l,s,1}\theta_{0,1,c}
        +
        \sum_{t=2}^{p-1}\Theta_{i,l,s,t}\theta_{p-t,t,c}
    }
    $$
    gives the number of times ${S_c}$ occurs as a composition factor of ${\Res^G_B(M)}$. Hence, ${\kappa_c = \mu_c - \theta_{a,L_i(l,s),c}}$ gives the number of times ${S_c}$ occurs as a composition factor of ${Q}$. The decomposition of $Q$ as a direct sum of indecomposable ${\mathbb{F}[B]}$ modules then follows from Proposition \ref{decomposition of Q}, as stated.
\end{proof}

%% file: bibliography.tex
\printbibliography[heading=none]